\documentclass[letterpaper,11 pt]{article}
\usepackage[margin=1in]{geometry}

\usepackage[authoryear]{natbib}

\usepackage{graphicx}

\usepackage{amsfonts}
\usepackage{bm} 
\usepackage{amsthm}
\usepackage{enumitem}

\usepackage{float} 
\usepackage{booktabs} 
\usepackage{multirow} 
\usepackage{algorithmic}
\usepackage{amsmath,amssymb}
\usepackage{bm}
\usepackage{tikz}
\usetikzlibrary{arrows.meta,positioning}
\usepackage{booktabs}
\usepackage[ruled,vlined]{algorithm2e}
\usepackage{empheq}  
\usepackage[most]{tcolorbox} %

\setcitestyle{maxnames=1}
\usepackage{color,soul}
\usepackage{nicefrac}

\usepackage{enumitem}

\newtcolorbox{ProblemBox}[1][]{%
  enhanced,
  breakable,
  colback=white,
  colframe=black,
  boxrule=0.6pt,
  left=1pt,right=1pt,top=-6pt,bottom=-4pt,
  sharp corners,
  #1
}

\newtcbox{\InnerMathBox}{%
  nobeforeafter,
  math upper,
  colback=white,
  colframe=black,
  boxrule=0.6pt,
  left=-2pt,right=0.1pt,top=2pt,bottom=-2pt,
  sharp corners
}
\usepackage{mdframed}
\usepackage{xcolor}
\usepackage{subcaption}

\usepackage{amsthm}
\newtheorem{theorem}{Theorem}
\newtheorem{assumption}{Assumption}
\usepackage[section]{placeins}

\newtheorem{lemma}{Lemma}

\usepackage[
    colorlinks=true,
    linkcolor=blue,   
    citecolor=black,  
    urlcolor=blue
]{hyperref}

\usepackage{amsmath,amssymb}
\usepackage{tikz}
\usetikzlibrary{
  positioning,
  calc,
  arrows.meta,
  shapes.misc,
  shapes.geometric,
  shadows.blur,
  decorations.markings,
  decorations.pathmorphing
}

\title{\textbf{Uncertainty-aware Power System Planning via Gradient Descent}}

\author{
Mehrnoush Ghazanfariharandi%
\thanks{Industrial and Systems Engineering, Rutgers University, NJ, USA.}
\and
Robert Mieth\footnotemark[1]
\textsuperscript{,}\thanks{Corresponding author. \newline 
Emails: \{mehrnoush.ghazanfariharandi, robert.mieth\}@rutgers.edu}
}

\date{}

\begin{document}

\maketitle

\begin{abstract}
Power system planning models provide important guidance on long-term investment strategies with significant socio-economic impact. To remain computationally manageable, however, such planning models compromise on the level of complexity with which power system operations and physics are captured. 
A common approach in most planning models is to collapse multi-stage power system operational processes into a single stage and, as a result, give up on the ability to account for uncertainty in each operational stage. 
In light of newly emerging load patterns and the continuing adoption of weather-dependent stochastic renewable generation, this uncertainty, however, becomes increasingly impactful on operations, and ignoring it has been shown to cause underinvestment in transmission capacity and flexible resources. In this work, we present a computational approach for power system expansion planning that explicitly considers two-stage day-ahead (DA) and real-time (RT) operational decisions under uncertainty while retaining time-coupling constraints to allow modeling generator ramping and energy storage. To solve the resulting optimization problem efficiently, we employ a projected stochastic gradient descent algorithm combined with a primal-dual optimization framework and an exponential moving average smoothing strategy to improve convergence stability. We evaluate the resulting investment decisions within a two-stage DA and RT simulation framework and compare them with a classic expansion planning model that assumes perfect knowledge of renewable generation. Our experiments show that the proposed framework achieves lower total system costs while ensuring that the implemented technology portfolio achieves set renewable integration targets.
\end{abstract}

\maketitle

\section{Introduction}

Demand for electric energy is rising world-wide.
In the United States, electric load is expected to grow by more than 50\% as transportation and industry electrify and as large-scale data centers expand \citep{NationalRenewableEnergyLaboratory2024annualtechnologybaseline,Shehabi2024DataCenterEnergy}. 
Meeting this growth will require major additions to generation, transmission, and energy storage infrastructure.
Such large-scale and long-term investments require dependable guidance from quantitative planning models on which investment pathways should be pursued.

Planning electric power infrastructure is complicated by power flow physics and high reliability requirements.
As a result, a central challenge for power system planning models is balancing modeling fidelity and computational tractability. 
In the past, power system planning models could rely on low-complexity proxies of power system operational practices and constraints due to the fact that traditional load patterns are well known and could be reasonably well predicted.
The validity of this assumption, however, is eroding due to the increasing adoption of new load types (\cite{force2025characteristics}) and the adoption of weather-dependent renewable generation, i.e., wind and solar PV.
The latter is accelerated by the competitive economics of renewable generation technologies (\cite{eia_lcoe_aeo2026}), and the presence of policies that prioritize their deployment to meet emission reduction goals in many jurisdictions.
Because power production from wind and solar power exhibits
high variability, limited predictability, and non-dispatchability, adequate expansion planning models must, on the one hand, capture the temporal dynamics of renewable production and, on the other hand, account for the fact that day-to-day system operations must deal with forecast errors arising from limited renewable predictability \citep{ela2014evolution,pineda2016impact}.
To date, there exists no power system planning model that adequately achieves both of these objectives. 

\subsection{Background and related literature}
The integration of renewable generation into long-term power system planning has been studied for several decades \citep{farghal1988generation,caramanis2010introduction} and there have been multiple studies on how large amounts of wind and solar power affect the optimal generation mix and how operational constraints such as unit commitment influence investment decisions \citep{de2011determining,pineda2016impact,jin2014temporal}. 
In fact, most existing models focus on capturing renewable variability \citep{arabali2014multi,munoz2013approximations,orfanos2012transmission} but neglect the role of forecast errors. As a result, they cannot capture how uncertainty during real-time operations changes the value of flexibility, for example, provided by battery storage, and the resulting incentives to invest in technologies that offer such flexibility. 
As a result, planning models may underestimate the value of flexibility and overestimate the efficiency of renewable energy utilization \citep{pineda2016impact}.

A line of research by \cite{pineda2016impact} and \cite{bylling2020impact} on planning models that do consider forecast errors, formulate system operations as a multi-stage clearing process, where a first day-ahead stage dispatches available generators based on point forecasts of load and renewable generation, and a second real-time stage balances forecast errors.
In particular, \cite{pineda2016impact} shows that optimal investments depend strongly on whether these stages are co-optimized in the planning model (assuming perfect knowledge of the system operator) or cleared sequentially to preserve the non-anticipatory nature of forecasts in real operations. 
Related work also emphasizes the importance of market design and uncertainty treatment for efficient investment and dispatch \citep{wang2008security,bessa2019handling,pritchard2010single}. 
Most recently, the work by \cite{mantegna2026uncertainty} has highlighted the limitations of deterministic capacity expansion models in addressing uncertainty and introduces a scalable two-stage adaptive robust optimization framework that explicitly incorporates forecast uncertainty into planning decisions. 
Their results underline the value of investing in flexible resources to make operations more robust and efficient.

With the rising need to counteract renewable variability and match renewable availability with load demand, battery storage will play a significant role in the technology portfolio of future power systems.
In this context, \cite{baker2016energy} explicitly highlights that net-load and renewable forecast uncertainty drive the optimal siting and sizing of storage. 
Indeed, state-of-the-art planning models co-optimize investments in and the chronological operation of battery storage and show that batteries can materially change optimal build decisions in systems with high renewable penetration \citep{levin2023energy,merrick2024representation,zhou2022joint,allahvirdizadeh2023stochastic,aguado2017battery}. 
However, integrating energy storage investments into planning models
introduces time coupling constraints through the state of charge, which complicates the necessary operations model and renders proposals to deal with forecast errors, such as \cite{pineda2016impact,bylling2020impact}, computationally intractable.

The computational challenge of representing power system operations that must deal with forecast errors and also keeping time-coupling operations constraints intact arises from the fact that separating operation decisions with non-anticipatory constraints from investment decisions creates a hard-to-solve bilevel problem structure \citep{pineda2016impact,degleris5169721gradient,mieth2026integrated}.
Traditionally, these problems are reformulated into single-level equivalent programs using optimality conditions that result in complex models with equilibrium constraints (MPECs).
Such approaches quickly become computationally intractable, especially when the operational layer includes richer dynamics or when large scenario sets are considered \citep{kleinert2023there,ye2019incorporating,ghazanfariharandi2025value,dempe2014kkt}. 
Recent methodological advances propose an alternative through gradient-based and first-order techniques that avoid explicit MPEC reformulations and scale better to large problems \citep{degleris5169721gradient,amos2017optnet,franceschi2018bilevel,ji2021bilevel}. In a similar spirit, \cite{kotary2025method} extends these ideas to bilevel problems with coupling constraints by combining differentiable lower-level solvers, projection operators, and gradient-based correction steps to ensure feasibility without relying on KKT-based single-level reformulations. 
These approaches, however, have not yet been applied to the problem that motivates this paper.

The recent work presented in \cite{mantegna2026uncertainty} is the most closely related to our paper in terms of motivation and problem setup. The authors also focus on capturing operational uncertainty and use an adaptive robust uncertainty set with a pre-defined budget to model imperfect knowledge during operations.
While similarly motivated, the resulting model in \cite{mantegna2026uncertainty} has two central distinctions from what we present in this paper. First, our approach does not require a pre-defined uncertainty set and budget, but is instead directly data-driven, using historical forecast and forecast-error data. Second, the approach in \cite{mantegna2026uncertainty} slightly alters current operational practice, since first-stage operations are aware of the uncertainty encoded in the uncertainty set and adjust their decisions accordingly. In current industry practice, only point forecasts are used, with little to no anticipation of real-time forecast errors.

\subsection{Contributions}
Motivated by the need for more accurate power system planning models in the presence of evolving load and renewable injection patterns, this paper proposes a power system planning problem alongside an effective solution method that (i) realistically models system operations as a two-step procedure with imperfect forecasts in the first stage and (ii) preserves time-coupling constraints in the operations model to capture energy storage and ramping constraints.
We solve this problem in a computationally tractable manner by avoiding an MPEC formulation of the operational problem and its optimality conditions. Instead, we express it as a sequence of optimization problems whose parameters are updated using projected stochastic gradient descent, implicit differentiation, and a primal–dual update framework.
We summarize or contributions as follows:
\begin{enumerate}[label=$\mathcal{C}$\arabic*]
    \item We formulate a power system planning problem that explicitly considers system operations as a cost-minimizing two-stage optimal power flow problem with imperfect knowledge on some parameters (e.g., wind power injections) in the first stage. 
    The operations problem is solved for a set of 24-hour days and modeled with relevant time coupling constraints to consider ramping and energy storage constraints.
    Relative to \cite{pineda2016impact}, this enables us to internalize the effects of forecast errors on optimal planning decisions that include battery storage investments.

    \item\label{cont:constraint} We use a projected stochastic gradient descent approach that leverages implicit differentiation of the optimization problems that model system operations. 
    Relative to \cite{degleris5169721gradient}, we model operations as a multi-stage process and we enforce a constraint (a renewable generation utilization target) in the planning problem that depends on the decisions of the operations problems. 
    We derive performance and convergence guarantees.

    \item We demonstrate the method on a numerical experiment and benchmark the resulting investment decisions against a standard investment model that assumes a perfect-knowledge system operator. 
    
\end{enumerate}

The presence of the renewable generation target (\ref{cont:constraint}) is noteworthy. The gradient descent method proposed in \cite{degleris5169721gradient} assumes that operational decisions only incur costs for the planner but do not capture the option of the planner wishing to achieve a target “behavior” of the system operator.
Enforcing such constraints is tricky within the proposed gradient descent approach, as the batched solution process prevents the straightforward conversion of the constraint into a soft constraint with a penalty in the objective. See also the discussion in \cite{wang2023learning}, who use a two-step Lagrangian relaxation approach to solve this problem.
In this paper, we propose a primal--dual framework to enforce the renewable target constraint dynamically during the solution process. 
The proposed solution approach supports large scenario sets and allows rapid updates when assumptions or data change.

\section{Problem formulation}
\label{sec:problemfor}
We take the perspective of a power system planner that wants to determine the amount and location of new generation, storage, and transmission capacity that minimizes the costs of investment and system operations while meeting pre-set goals on the generation portfolio. 
In the following, we present the resulting formulation in consecutive levels of detail to introduce the relevant modeling concepts.
Section~\ref{ssec:investment_planning_with_decoupled_operations_decisions} first formulates a general planning problem where system operations independently respond to planning decisions. 
Section~\ref{ssec:multi_stage_operations} then introduces the separation of system operations in multiple stages with imperfect knowledge on which we build a more detailed planning model in Section~\ref{ssec:planning_with_multi_stage_operations_and_operations_constraints}, that also includes renewable integration targets. 
This formulation will be the basis for discussing our solution approach and formal results.
Finally, Section~\ref{ss:planningandoperation} presents all details of the investment and operations model.

\subsection{Planning with decoupled operations decisions}
\label{ssec:investment_planning_with_decoupled_operations_decisions}

The planner seeks to optimize a planning objective $J(\bm{\theta})$ by choosing investment decisions 
$\bm{\theta}$ within a feasible set of investments $\mathcal{H}$.
In this general form, objective $J$ may minimize the total cost of constructing new physical assets, but may also capture other planning goals such as ensuring system reliability, minimizing $\mathrm{CO}_2$ emissions, or maximizing renewable energy integration. 
Notably, system operations depend on investment decisions $\bm{\theta}$, but the planner typically does not operate the system and/or cannot alter operational processes. 
Therefore, the planner must consider that system operations focus on their own objective, which may not necessarily be perfectly aligned with the planning objective. 
For example, the planner may wish to achieve a target level of low-emission generation, but the system operator may only focus on achieving cost-minimal day-to-day operations with the available resources $\bm{\theta}$.
To capture this complication, we write the planning problem subject to the decisions of an ``inner'' operations problem as:
\begin{subequations}
\label{planningproblem}
\begin{align}
\min_{\bm{\theta}} \quad
J(\bm{\theta})
&:= \bm C^{\mathrm{Inv}}(\bm{\theta})
+ \sum_{d\in[D]} \pi_d  h_d\!\left(\bm x^*_{d}(\bm{\theta})\right) \label{eq:obj_firstproblem}\\
\text{s.t.}\quad
&\bm{\theta} \in \mathcal{H}, \label{invconss}\\
&\bm x^*_{d}(\bm{\theta})
=
\arg\min_{\bm{x}}
\left\{
\bm C^{\mathrm{Opr}}\!\left(\bm x,\bm{\theta},\bm{\xi}_d\right)\quad
\ \text{s.t.}\ \quad
\bm x \in \mathcal{X}(\bm{\theta},\bm{\xi}_d)
\right\},
\quad \forall d\in[D]. \label{operproblem}
\end{align}
\end{subequations}
The problem is formulated over a set of $D$ scenarios, e.g., 24-hour days, indexed by $d$. 
The formulation in \eqref{planningproblem} accommodates the common practice of using ``representative days'' with a potentially non-uniform weight $\pi_d$. For the remainder of this paper we assume scenarios are drawn from a set of historical days with equal weight.
The inner problem is a cost-minimization problem that reflects generator dispatching or electricity market clearing procedures.
Function $h$ maps the operational decisions $\bm{x}$ into the objective of the planner.
These decisions $\bm{x}$ include, for example, generator schedules, storage charging/discharging actions, and transmission line flows.
The feasible operational set $\mathcal{X}(\bm{\theta},\bm{\xi}_d)$
depends on investments $\bm{\theta}$ and scenarios $\bm{\xi}_d$ of external parameters, e.g., load demand or renewable power production, that are available for all $d\in[D]$ scenario days.
Thus, set $\mathcal{X}$ collects the system’s physical and operational constraints, including capacity limits on generators, storage, and transmission, ramping limits and other intertemporal constraints, network power-flow constraints and transmission limits, and nodal power balance.
We discuss details of these constraints in Section~\ref{ss:planningandoperation} below.
The operations cost function $\bm C^{\mathrm{Opr}}$ captures short-term production and balancing costs under these
constraints.
For the problem in \eqref{planningproblem} to be well-posed, some properties of the inner problem \eqref{operproblem} such as feasibility and uniqueness of the solution are implied. 
We discuss these assumptions in detail in Section~\ref{sec:Convergence Analysis} below.

The problem formulation in \eqref{planningproblem} explicitly decouples the operations decisions from the planning objective, leading to the bilevel problem structure of \eqref{planningproblem}.
We note that if the objective of the planner and the operations do happen to align, for example if both only care about cost minimization, \textit{and} there is no uncertainty in the operations stage, the bilevel structure would not be necessary and the problem can be formulated as a single-level, albeit potentially large-scale, optimization problem. 
The presence of imperfect information, however, further motivates formulating the problem in a bilevel structure as we discuss in the following.

\subsection{Multi-stage operations with imperfect knowledge}
\label{ssec:multi_stage_operations}

In practice, power system operations rely on a multi-stage decision-making sequence. For example, a day-ahead (DA) and a real-time (RT) market clearing stage.
We will adopt this DA+RT setup for the remainder of the paper as it is the most common power system organization in the U.S. and Europe.
The DA stage schedules resources based on forecasts of uncertain quantities such as weather-dependent renewable generation and demand, while the RT stage corrects for forecast errors once uncertainty is realized. 
Therefore, the resulting discrepancy between DA schedules and RT realizations directly impacts operating costs, reliability, congestion patterns, and the value of flexibility.

To capture this, we expand the operational model \eqref{operproblem} into two stages. 
Given a point forecast $\bm{\bar{\xi}}$ of the uncertain parameters $\bm{\xi}$, the DA stage selects a schedule $\bm x^{DA}$:
\begin{equation}
\label{da}
\text{DA stage:}\quad
\min_{\bm x^{DA}}\left\{
\bm C^{\mathrm{Opr,DA}}(\bm x^{DA})
\quad\ \text{s.t.}\ \quad
\bm x^{DA} \in \mathcal{X}^{\rm DA}(\bm\theta,\bm{\bar\xi})
\right\}.
\end{equation}
After uncertainty realizes as $\tilde{\bm\xi}$, the RT stage selects recourse actions
$\bm x^{RT}$ (e.g., redispatch):
\begin{equation}
\label{rt}
\text{RT stage:}\quad
\min_{\bm x^{RT}}\left\{
\bm C^{\mathrm{Opr,RT}}(\bm{x}^{DA,*},\bm x^{RT})
\ \quad \text{s.t.}\ \quad
\bm x^{RT} \in \mathcal{X}^{\rm RT}(\bm\theta,\bm{x}^{DA,*},\tilde{\bm\xi})
\right\},
\end{equation}
where $\mathcal{X}^{\rm RT}(\cdot)$ enforces real-time feasibility constraints and limits corrective actions based on previous decisions $\bm{x}^{DA,*}$.
Given this two-stage model, each scenario $\bm{\xi}_d$ therefore contains two components:
(i) a forecast of uncertain parameter, denoted by $\bar{\bm{\xi}}_d$, which is used in the DA market clearing, and
(ii) the realized uncertain parameter, denoted by $\tilde{\bm{\xi}}_d$, which materializes in real time and determines the RT balancing actions. We write
$
\bm{\xi}_d = \left(\bar{\bm{\xi}}_d,\tilde{\bm{\xi}}_d\right)
$.

The presence of a DA and RT stage with imperfect foresight raises the question to what extend the degree of uncertainty and coordination between them affects power system expansion planning. 
Intuitively, with increasing uncertainty of resource availability in the RT stage during DA decision-making, e.g., due to increasing net-load uncertainty from higher shares or renewable generation capacity, the impact of DA decisions $\bm{x}^{DA,*}$ on the RT solution $\bm{x}^{RT,*}$ intensifies.
The objective of our work is to study this impact in the context of a planning model.
We note that traditional planning models essentially assume that all information is available at the DA stage, i.e., that there is no uncertainty. Under this assumption, the operations problem can be modeled as a single stage. 
If the planner then also assumes that planning and operations objectives align, as discussed above, the model collapses to a single-level optimization problem, which we will also use as a reference as discussed in Section~\ref{ssec:testing_system} below.

\subsection{Planning with multi-stage operations and operations constraints}
\label{ssec:planning_with_multi_stage_operations_and_operations_constraints}

We now formulate a more detailed version of \eqref{planningproblem} that will be the basis for deriving our proposed solution method.
To this end, we also introduce a final important consideration.
So far, similar to the model proposed in \cite{degleris5169721gradient}, our formulation in \eqref{planningproblem} cannot impose targets on the decisions made in the operations problem. 
However, practical planning usually has performance targets such as achieving a minimum level of renewable generation. 
The following model enforces such a constraint and introduces the specific model structure that we will use to derive a tractable solution approach and discuss its performance:

\begin{subequations}\label{tri-level}
\begin{align}
&\min
&& J(\bm\theta)
:= \sum_{y \in [Y]} \Big(\frac{1}{1+r}\Big)^{\tau(y)} \Bigg(
C_y^{\mathrm{inv}}(\bm\theta_y)
+ \sum_{d=1}^{D}\Big( \bm C_{d,y}^{\text{DA}}\big(
\bm x_{d,y}^{\rm{DA},*}(\bm\theta_y)\big) +
\bm C_{d,y}^{\text{RT}}\big(
\bm x_{d,y}^{\rm{RT},*}(\bm\theta_y,\bm x_{d,y}^{\rm{DA},*}(\bm\theta_y))\big)
\Big) \Bigg)
\label{obj_tri-level}\\[0.1em]
&\text{s.t.} && \underline{\bm\theta}_y \le \bm\theta_y \le \overline{\bm\theta}_y,
\qquad \forall y \in [Y],
\label{limit_inv}\\
&&& \bm\theta_{y-1} \le \bm\theta_y,
\qquad \forall y \in [Y]\setminus\{1\},
\label{monocity_inv}\\
&&&
\eta_y\,\left(\sum_{d \in [D]} \Big(G^{\mathrm{conv}}\!\left(
\bm x_{d,y}^{\rm{DA},*} , \bm x_{d,y}^{\rm{RT},*}
\right) + G^{\mathrm{ren}}\!\left(
\bm x_{d,y}^{\rm{RT},*}
\right) \Big)\right) 
\le 
\sum_{d \in [D]} \Big(G^{\mathrm{ren}}\!\left(
\bm x_{d,y}^{\rm{RT},*}
\right)\Big),
\qquad  \forall y \in [Y],
\label{outer_renewable_target}\\[0.2em]
&&&
\bm x_{d,y}^{\rm{DA},*}(\bm\theta_y)
\in
\left\{
\begin{aligned}
\arg\min_{\bm x_{d,y}^{\rm{DA}}}\;& C_{d,y}^{\rm{DA}}(\bm x_{d,y}^{\rm{DA}})\\
\text{s.t.}\quad
& A_{d,y}^{\rm{DA}}(\bm\theta_y)\,\bm x_{d,y}^{\rm{DA}}
\le \bm b_{d,y}^{\rm{DA}}(\bm\theta_y, \bm{\bar \xi}_{d,y}).
\end{aligned}
\right\}
\ \forall d \in [D],\forall y \in [Y]
\label{middlelevel}
\\[-0.2em]
&&&
\bm x_{d,y}^{\rm{RT},*}\bigl(\bm\theta_y,\bm x_{d,y}^{\rm{DA},*}(\bm\theta_y)\bigr)
\in
\left\{
\begin{aligned}
\arg\min_{\bm x_{d,y}^{\rm{RT}}}\;& C_{d,y}^{\rm{RT}}(\bm x_{d,y}^{\rm{RT}})\\
\text{s.t.}\quad
& A_{d,y}^{\rm{RT}}\bigl(\bm\theta_y,\bm x_{d,y}^{\rm{DA},*}(\bm\theta_y)\bigr)\,\bm x_{d,y}^{\rm{RT}}
\le \bm b_{d,y}^{\rm{RT}}\bigl(\bm\theta_y,\bm x_{d,y}^{\rm{DA},*}(\bm\theta_y), \tilde{\bm \xi}_{d,y}\bigr)
\end{aligned}
\right\} \label{innerlevel}\\
&&&
\hspace{10cm} \forall d \in [D],\forall y \in [Y]. \nonumber
\end{align}
\end{subequations}
For a given set of years indexed by $y\in[Y]$, objective \eqref{obj_tri-level} computes the net-present value of the investment cost and the operations cost arising at the DA stage \eqref{middlelevel} and RT stage \eqref{innerlevel} using discount factor $r$ and the number of years $\tau(y)$  between a reference year (e.g., ``today'') and the year with index $y$.
To define the feasible investment space, constraint \eqref{limit_inv} limits the capacity potential of each technology and constraint \eqref{monocity_inv} imposes intertemporal monotonicity so that capacity can only increase across years\footnote{ 
We note that this constraint does not prevent modeling pre-defined decommissioning of generators, which can be captured by an adjustment of available capacities in the operations models. 
}.
Together, constraints \eqref{limit_inv} and \eqref{monocity_inv} represent the investment space $\mathcal{H}$ from \eqref{invconss}. 
In addition, constraint \eqref{outer_renewable_target} imposes the annual system-wide renewable energy target. 
For each year $y$, the parameter $\eta_y \in [0,1]$ specifies the minimum fraction of total annual energy generation that renewable resources should supply. 
Here, $G^{\rm conv}$ and $G^{\rm ren}$ are shorthand for extracting the total conventional and renewable energy production from the operational decisions.
Notably, renewable targets are not enforced during daily operations, which usually just aim for a cost-minimal economic dispatch. 
We indicate this by writing constraint \eqref{outer_renewable_target} to explicitly depend on the optimal operational decisions 
$\bm x_{d,y}^{\rm{DA},*}\big(\bm \theta_y\big)$ and 
$\bm x_{d,y}^{\rm{RT},*}\!\big(\bm \theta_y, \bm x_{d,y}^{\rm{DA},*}(\bm\theta_y)\big)$ of the DA and RT problems in  \eqref{middlelevel} and \eqref{innerlevel}.

We model the DA and RT problems as linear programs. This is a common assumption as most technical constraints directly admit a linear representation and bulk power system power flow physics are well-represented by linear approximations (e.g., ``DC'' power flow). The detailed formulations are shown in Section~\ref{ss:planningandoperation} below. 
In the compact form here, the matrix \(A_{d,y}^{\rm{DA}}(\bm{\theta}_y)\) in \eqref{middlelevel} collects the coefficients of all linear DA constraints and the vector \(b_{d,y}^{\rm{DA}}(\bm{\theta}_y,\bar{\bm{\xi}}_{d,y})\) contains the corresponding right-hand-side values, such as electricity demand, renewable generation forecasts, and capacity bounds implied by a given investment decision \(\bm{\theta}_y\) and forecast $\bar{\bm{\xi}}_{d,y}$. 
Using these constraints, the DA objective function \(C_{d,y}^{\rm{DA}}(\bm{x}_{d,y}^{\rm{DA}})\) minimizes generation costs, storage operating costs, and penalty terms associated with load shedding and renewable curtailment.
Similarly,  matrix \(A_{d,y}^{\rm{RT}}(\bm{\theta}_y,\bm{x}_{d,y}^{\rm{DA},*})\) includes the coefficients of all RT constraints and the vector \(b_{d,y}^{\rm{RT}}(\bm{\theta}_y,\bm{x}_{d,y}^{\rm{DA},*},\tilde{\bm{\xi}}_{d,y})\) defines the corresponding RT right-hand-side values under the realized uncertainty outcomes $\tilde{\bm{\xi}}_{d,y}$ and the decision of the DA stage $\bm{x}_{d,y}^{\rm{DA},*}$.
Based on these constraints, the RT objective function \(C_{d,y}^{\rm{RT}}(\bm{x}_{d,y}^{\rm{RT}})\) minimizes balancing costs, including upward and downward re-dispatch costs and penalty terms for renewable curtailment and load shedding.

The proposed formulation in \eqref{tri-level} will be central to our discussion in Section~\ref{sec:solution_methodology} on how to solve this problem efficiently and to discuss formal results. 
The following section now presents modeling details.

\subsection{Detailed planning and operation problems formulation}
\label{ss:planningandoperation}

We first present the detailed planning constraints and objective followed by details on the objectives and constraints used for the DA and RT operations models. 
Without loss of generality, we focus our discussion on the modeling of \textit{wind power} as a stand-in for general renewable injections and/or net-load uncertainty.

\subsubsection{Planning problem}
 
Vector $\bm{\theta}_y$ collects the investment variables for each year
$\bm \theta_y := \{\bar{\bm p}_y,\, \bar{\bm f}_y,\, \bar{\bm p}^{B}_y,\, \bar{\bm E}_y,\, \bar{\bm p}^{w}_y\}$, i.e., installed capacities of conventional generation, transmission lines, battery power, battery energy, and wind generation over the planning horizon.
These variables are additions to existing assets that enter the model as fixed parameters.
Hence, we define the planning objective $C_y^{\mathrm{inv}}$ as 
\begin{subequations}\label{invmodel}
\begin{equation}
    C_y^{\mathrm{inv}}(\bm{\theta}_y) = \sum_{l \in [L]} c^{\rm L}_{l,y}\,\bar f_{l,y}
+ \sum_{g \in [G]} c^{\rm G}_{g,y}\,\bar p_{g,y}
+ \sum_{b \in [B]} c^{\rm E}_{b,y}\,\bar E_{b,y}
+ \sum_{w \in [W]} c^{\rm W}_{w,y}\,\bar p^{w}_{w,y},
\label{objplanning}
\end{equation}
with investment cost coefficients $c^{\rm L}_{l,y}, c^{\rm G}_{g,y}$, $c^{\rm E}_{b,y}$, $c^{\rm W}_{w,y}$.
We note here that this strictly linear approach to modeling investment costs is a simplifying assumption as practical investments in generation and transmission capacity are usually ``lumpy'' and invoke fixed costs that are independent of the installed capacity. 
We consider tackling this problem beyond the scope of this particular work.
See also the discussion on Assumption~\ref{assumption1} in Section~\ref{sec:Convergence Analysis} below.

The detailed investment constraints (constraints \eqref{limit_inv} and \eqref{monocity_inv} above) are:
\begin{align}
& 0 \le \bar f_{l,y} \le \bar f^{\max}_{l},
&& \forall l \in [L],\ \forall y \in [Y]
\label{Line capacity limits}
\\
& 0 \le \bar p_{g,y} \le \bar p^{\max}_{g},
&& \forall g \in [G],\ \forall y \in [Y]
\label{Generation output limits.}
\\
& 0 \le \bar E_{b,y} \le \bar E^{\max}_{b},
&& \forall b \in [B],\ \forall y \in [Y]
\label{Storage energy bounds.}
\\
& 0 \le \bar p^{w}_{w,y} \le \bar p^{\rm W,\max}_{w},
&& \forall w \in [W],\ \forall y \in [Y]
\label{Renewable installation bounds}
\\
& \bar f_{l,y-1} \le \bar f_{l,y},
&& \forall l \in [L],\ \forall y \in [Y]\setminus\{1\}
\label{Inter-temporal monotonicity of line expansion.}
\\
& \bar p_{g,y-1} \le \bar p_{g,y},
&& \forall g \in [G],\ \forall y \in [Y]\setminus\{1\}
\label{Inter-temporal monotonicity of generation capacity.}
\\
& \bar E_{b,y-1} \le \bar E_{b,y},
&& \forall b \in [B],\ \forall y \in [Y]\setminus\{1\}
\label{Inter-temporal monotonicity of energy capacity.}
\\
& \bar p^{w}_{w,y-1} \le \bar p^{w}_{w,y},
&& \forall w \in [W],\ \forall y \in [Y]\setminus\{1\}
\label{Renewable installation monotonicity.}
\\
& \bar E_{b,y} = 4\,\bar p^{B}_{b,y},
&& \forall b \in [B],\ \forall y \in [Y].
\label{Energy-to-power ratio.}
\end{align}
\end{subequations}
Constraints \eqref{Line capacity limits}--\eqref{Renewable installation bounds} set the expansion limits to transmission, conventional generation, battery storage, and wind generation. 
We characterize battery systems through energy capacity $\bar E_{b,y}$ and power rating $\bar p^{B}_{b,y}$ which are coupled by a fixed energy-to-power ratio as given in \eqref{Energy-to-power ratio.}.
We focus here on 4-hour storage, i.e., batteries with a capacity that supports the maximum discharging power for four hours. 
Constraints \eqref{Inter-temporal monotonicity of line expansion.}--\eqref{Renewable installation monotonicity.} impose inter-temporal monotonicity and keep installed capacities for transmission lines, conventional generation, battery energy, and wind power non-decreasing across years.

\subsubsection{DA operations}

The DA stage in \eqref{middlelevel} determines a cost-minimal generator dispatch given a forecast of uncertain wind power availability. Without loss of generality, we assume all other parameters to be known. 
The resulting DA objective is:
\begin{subequations}\label{daprob}
\begin{align}
    C_{d,y}^{\rm DA} = \sum_{t \in [T]} \Bigg(
\sum_{g \in [G]} c^{\rm G}_g\, p_{g,d,t,y}
+ \sum_{b \in [B]} c^{\rm B}_b(p^{dis,DA}_{b,d,t,y}  +p^{ch,DA}_{b,d,t,y})
+ \sum_{i \in [N]} c^{\rm shed}\, d^{shed,DA}_{i,d,t,y}
+ \sum_{w \in [W]} c^{\rm cur}\, w^{cur,DA}_{w,d,t,y}
\Bigg).
\label{eq:daobj}
\end{align}
The DA dispatch is solved for each day $d$ and year $y$ for $T=24$ one-hour timesteps indexed as $t\in[T]$. 
The objective is the sum of the production cost of conventional generators with cost coefficients $c^{\rm G}_g$, cost from charging and discharging of battery storage with cost coefficient $c^{\rm B}_b$, and costs for renewable curtailments and load shedding with cost coefficients $c^{\rm cur}$ and $c^{\rm shed}$, respectively. 
The DA dispatch must ensure power balance at each node $i$ of the network:
\begin{align}
\label{danodal} 
\!\!\!
\sum_{g \in [G]_i} p_{g,d,t,y}
+ \!\!\sum_{b \in [B]_i}\! p^{B,DA}_{b,d,t,y}
+ \!\!\sum_{l \in \delta_i^{in}}\! f^{DA}_{l,d,t,y}
- \!\!\sum_{l \in \delta_i^{out}}\! f^{DA}_{l,d,t,y}
+ \!\!\sum_{w \in [W]_i}\! p^{w,DA}_{w,d,t,y}
+ d^{shed,DA}_{i,d,t,y}
= d_{i,d,t,y},\nonumber\\\quad \forall i \in [N],\ \forall t \in [T],
\end{align}
i.e., at every node $i$, conventional generation $p_{g,d,t,y}$, storage injection $p^{B,DA}_{b,d,t,y}$, incoming and outgoing line flows $f^{DA}_{l,d,t,y}$, scheduled wind production $p^{w,DA}_{w,d,t,y}$, and load shedding $d^{shed,DA}_{i,d,t,y}$ must meet demand $d_{i,d,t,y}$.
We use $[G]_i$ as a shorthand for the index set of generators connected to bus $i$ (analogously for the set of battery storage $[B]$ and wind generators $[W]$).
Variables $f^{DA}_{l,d,t,y}$ denote the active power flow on transmission lines and we write $\delta_i^{in}$, $\delta_i^{out}$ for the index set of transmission lines that are ``entering'' or ``leaving'' bus $i$, respectively, as per their pre-defined orientation.

Conventional generator production $p_{g,d,t,y}$ is constrained by the installed capacity, and the upward and downward ramping constraints of the generators:
\begin{align}
& 0 \le p_{g,d,t,y} \le \bar p_{g,y},
&& \forall g \in [G],\ \forall t \in [T]
\label{Generation output limits.da}
\\
& p_{g,d,t+1,y} - p_{g,d,t,y} \le R^{\rm up}_{g},
&& \forall g \in [G],\ \forall t=1,\dots,T-1
\label{rampup}
\\
& p_{g,d,t,y} - p_{g,d,t+1,y} \le R^{\rm dn}_{g},
&& \forall g \in [G],\ \forall t=1,\dots,T-1
\label{rampdown}
\end{align}

We model power flow using the common DC power flow approximation in voltage angle notation. To this end, each bus $i$ is characterized by its voltage angle $\theta^{DA}_{i,d,t,y}$, which is set to  $\theta^{DA}_{{\rm ref},d,t,y} = 0$ for a reference bus $i={\rm ref}$.
The resulting power flow is then a linear mapping from the voltage angle difference of the buses to which line $l$ is connected to and the line's electrical properties. See also \cite{weinhold2020fast}. We capture this mapping as  $H$ and write:
\begin{align}
&f^{DA}_{l,d,t,y}
= \sum_{i \in [N]} H_{l,i}\theta^{DA}_{i,d,t,y},
&& \forall l \in [L],\ \forall t \in [T]
\label{transmissionlinkda}
\\
& \theta^{DA}_{\rm{ref},d,t,y} = 0,
&& \forall t \in [T]
\label{transmissionlinkda_zero}
\\
& -\bar f_{l,y} \le f^{DA}_{l,d,t,y} \le \bar f_{l,y},
&& \forall l \in [L],\ \forall t \in [T].
\label{Line capacity limitsda}
\end{align}
Constraint \eqref{Line capacity limitsda} bounds line flows by the installed transmission capacities $\bar f_{l,y}$.

We assume ideal storage and model their operation as:
\begin{align}
& -\bar p^{B}_{b,y} \le p^{B,DA}_{b,d,t,y} \le \bar p^{B}_{b,y},
&& \forall b \in [B],\ \forall t \in [T]
\label{Storage power limits.da}
\\
& e^{DA}_{b,d,t+1,y} = e^{DA}_{b,d,t,y} - p^{B,DA}_{b,d,t,y},
&& \forall b \in [B],\ \forall t=1,\dots,T-1
\label{Storage energy transition.da}
\\
& 0 \le e^{DA}_{b,d,t,y} \le \bar E_{b,y},
&& \forall b \in [B],\ \forall t \in [T]
\label{Storage energy bounds.da}
\\
& 0 \le p^{dis,DA}_{b,d,t,y},
&& \forall b \in [B],\ \forall t \in [T]
\label{Storage power obj-variable}
\\
& 0 \le p^{ch,DA}_{b,d,t,y},
&& \forall b \in [B],\ \forall t \in [T]
\label{Storage power obj-variable2}
\\
& p^{B,DA}_{b,d,t,y} = p^{dis,DA}_{b,d,t,y} - p^{ch,DA}_{b,d,t,y},
&& \forall b \in [B],\ \forall t \in [T].
\label{Storage power obj}
\end{align}
Constraint \eqref{Storage power limits.da} limits the net power output $p^{B,DA}_{b,d,t,y}$ of each storage unit. 
Constraint \eqref{Storage energy transition.da} governs the intertemporal dynamics of stored energy $e^{DA}_{b,d,t,y}$, and constraint \eqref{Storage energy bounds.da} bounds stored energy by the installed energy capacity $\bar E_{b,y}$. Constraints \eqref{Storage power obj-variable} and \eqref{Storage power obj-variable2} enforce nonnegativity of discharging and charging power, and constraint \eqref{Storage power obj} defines net power output as the difference between discharging and charging decisions.

Scheduled wind production $p^{w,DA}_{w,d,t,y}$ is limited by the available DA wind availability forecast $\bar{\xi}_{w,d,t,y}\in[0,1]$ and the installed wind capacity $\bar p^w_{w,y}$ and any unused wind production potential is defined as wind curtailment $w^{cur,DA}_{w,d,t,y}$:
\begin{align}
& 0 \le p^{w,DA}_{w,d,t,y} \le \bar \xi_{w,d,t,y}\,\bar p^w_{w,y},
&& \forall w \in [W],\ \forall t \in [T]
\label{windda}
\\
& w^{cur,DA}_{w,d,t,y}
= \bar \xi_{w,d,t,y}\,\bar p^w_{w,y} - p^{w,DA}_{w,d,t,y},
&& \forall w \in [W],\ \forall t \in [T].
\label{Corrective action limits.curda}
\end{align}

Finally, any load shedding is nonnegative and limited to the load demand:
\begin{align}
    & 0 \le d^{shed,DA}_{i,d,t,y} \le d_{i,d,t,y},
&& \forall i \in [N],\ \forall t \in [T].
\label{Corrective action limits.da}
\end{align}
\end{subequations}

\subsubsection{RT operations}

During RT operation, the system operator observes the realized wind power injection and resolves the resulting imbalances by finding a cost minimal re-dispatch of conventional generators and deciding on final battery charging and discharging, wind spillage, and load curtailment.
We use $p^{+}_{g,d,t,y}$ and $p^{-}_{g,d,t,y}$ to denote upward and downward balancing of generator $g$ at cost $c^{\rm up}_g$ and $c^{\rm dn}_g$, respectively. All other variables are defined analogously to the DA formulation above, but are decorated with a superscript $RT$. The resulting RT objective function is:
\begin{subequations}
\label{RTprob}
\begin{align}
&C_{d,y}^{\rm RT} =  \!\sum_{t\in[T]}\!\Bigg(\!\sum_{g\in[G]}(c^{\rm up}_g\, p^+_{g,d,t,y} \!- c^{\rm dn}_g\, p^-_{g,d,t,y}) +\!\! \sum_{i\in[N]} (c^{\rm shed}\, d^{shed,RT}_{i,d,t,y} + c^{\rm cur}\, w^{cur,RT}_{i,d,t,y}) +\!\! \sum_{b\in[B]} c^{\rm B}_b (p^{dis,RT}_{b,d,t,y} +\nonumber\\& p^{ch,RT}_{b,d,t,y})\Bigg).
\end{align}

The real-time nodal power balance enforces equality of supply and demand with respect to the DA generator schedule and the realized wind power availability $\tilde{\bm{\xi}}$, which limits the usable RT wind power injection $p_{w,t,t,y}^{w,RT}$:
\begin{align}
& \sum_{g \in [G]_i}\!\! \left(p_{g,d,t,y} + p^+_{g,d,t,y} \!\!- p^-_{g,d,t,y}\right)\! +\!\! \sum_{b \in [B]_i} p^{B,RT}_{b,d,t,y}\!
+\!\! \sum_{l \in \delta_i^{in}}\! f^{RT}_{l,d,t,y} -\!\! \sum_{l \in \delta_i^{out}} f^{RT}_{l,d,t,y} 
+\!\! \sum_{w \in [W]_i}\! p^{w,RT}_{w,d,t,y}
+ \nonumber\\&d^{shed,RT}_{i,d,t,y}
= d_{i,d,t,y}, \hspace{9cm}\forall i \in [N],\ \forall t \in [T].
\label{rt_first1}
\end{align}
All variables and sets in \eqref{rt_first1} are defined analogously to their DA equivalent in \eqref{danodal}. 
For a given DA dispatch decision $p_{g,d,t,y}$ generator re-dispatch in RT is limited by:
\begin{align}
& 0 \le p_{g,d,t,y} + p^+_{g,d,t,y} - p^-_{g,d,t,y} \le \bar p_g,
&& \forall g \in [G],\ \forall t \in [T]
\label{Generation output limits.rt}
\\
& \left(p_{g,d,t+1,y} + p^+_{g,d,t+1,y}\! - p^-_{g,d,t+1,y}\right)\!-\! \left(p_{g,d,t,y} + p^+_{g,d,t,y}\! - p^-_{g,d,t,y}\right)
\!\le R^{\rm up}_g,
&& \forall g \in [G],\ \forall t=1,\dots,T\!-\!1
\label{rampup_rt}
\\
& \left(p_{g,d,t,y} + p^+_{g,d,t,y}\! - p^-_{g,d,t,y}\right)\! -\! \left(p_{g,d,t+1,y} + p^+_{g,d,t+1,y} \!- p^-_{g,d,t+1,y}\right)\!
\le R^{\rm dn}_g,
&& \forall g \in [G],\ \forall t=1,\dots,T\!-\!1,
\label{ramprt}
\end{align}
and RT wind power is constrained by
\begin{align}
& 0 \le p^{w,RT}_{w,d,t,y} \le \tilde{\xi}_{w,d,t,y}\,\bar p^w_{w,y},
&& \forall w \in [W],\ \forall t \in [T]
\label{windda_rt}
\\
& w^{cur,RT}_{w,d,t,y}
= \tilde{\xi}_{w,d,t,y}\,\bar p^w_{w,y} - p^{w,RT}_{w,d,t,y},
&& \forall w \in [W],\ \forall t \in [T].
\label{Corrective action limits.currt}
\end{align}
\end{subequations}
Otherwise, we enforce the same constraints on power flow, transmission line limits, and battery storage dynamics in RT as we do in DA, with the only difference that RT variables are identified by their $RT$ superscript.

With all the DA and RT variables introduced, we can now define the maps $G^{\rm conv}$ and $G^{\rm ren}$ from \eqref{outer_renewable_target} as:
\begin{align}
    G^{\rm conv} = \sum_{t \in [T]} \Bigg(\sum_{g \in [G]}
\left(
p_{g,d,t,y}
+
p^+_{g,d,t,y}
-
p^-_{g,d,t,y}
\right)\Bigg), \quad G^{\rm ren} = \sum_{t \in [T]} \Bigg( \sum_{w \in [W]} p^{w,RT}_{w,d,t,y}\Bigg).
\end{align}

\section{Solution methodology} 
\label{sec:solution_methodology}

Our problem as formulated in \eqref{tri-level} is hard to solve. 
It is a multi-level program that scales with the number of scenario days and that eventually should be usable at the scale of practical expansion planning models.
This disqualifies the classic solution approach of including the RT problem into the outer level, retaining the DA problem as an inner level problem and formulating a single-level equivalent of the resulting bilevel problem using the inner problem's optimality conditions. 
This approach was pursued in \cite{pineda2016capacity,bylling2020impact}.
The resulting problem is an MPEC that requires additional reformulations leading to a large-scale mixed-integer problem. 
Even for moderately sized planning problems, the presence of time-coupling constraints renders this solution approach impractical \citep{pineda2016impact}.
Instead, we use a gradient-based solution method that leverages implicit differentiation of the inner optimization problems.
Our approach is similar to \cite{degleris5169721gradient} but additionally enforces the constraint on renewable generation as discussed in Section~\ref{ssec:planning_with_multi_stage_operations_and_operations_constraints} above.
Also in our problem \eqref{tri-level}, the upper-level objective depends not only on the investment vector $\bm{\theta}$, but also explicitly on the DA and RT operating decisions \textit{and} RT operations depend on DA outcomes. 
As a result, computing gradients of the outer objective with respect to $\bm{\theta}$ requires tracking both the direct impact of investments and the indirect impact transmitted through the DA problem. 
In the following, we derive the required gradients with respect to $\bm{\theta}$ and develop a projected stochastic gradient descent (SGD) procedure to solve the expansion planning problem \eqref{tri-level} efficiently.
Fig.~\ref{fig:tri_level_scheme} provides a schematic overview of the solution process.

\subsection{Projected stochastic gradient descent}
Each iteration $(k)$ of the SGD algorithm starts with a current vector of investment decisions $\bm{\theta}^{(k)}_y$ for each year $y \in [Y]$.
To evaluate this current decision iterate, we sample a subset (batch) of $B$ days $\mathcal B_{y}^{(k)} \subseteq [D]$. For these days we then solve the corresponding DA and RT models and evaluate the outer objective $J\big(\bm{\theta}^{(k)}_1,\dots,\bm{\theta}^{(k)}_Y\big)$ \textit{and its gradients} $\nabla_{\bm{\theta}^{(k)}_1}J,\dots,\nabla_{\bm{\theta}^{(k)}_Y}J$. 
To this end, 
for each sampled day $d \in \mathcal B_{y}^{(k)}$ in year $y$, we parametrize the DA and RT models with the respective renewable forecasts $\bar{\bm{\xi}}_{d,y}$, the actual renewable availability $\tilde{\bm{\xi}}_{d,y}$, the nodal demand $d_{i,d,t,y}$, and the current investments $\bm{\theta}_y^{(k)}$.
Recall that $\bm x_{d,y}^{\rm{DA},*}(\bm{\theta}^{(k)}_y)$ and $\bm x_{d,y}^{\rm{RT},*}(\bm{\theta}^{(k)}_y)$ are the optimal primal solutions for day $d$ and year $y$, and let $\bm\lambda_{d,y}^{\rm{DA},*}(\bm{\theta}^{(k)}_y)$ and $\bm\lambda_{d,y}^{\rm{RT},*}(\bm{\theta}^{(k)}_y)$ denote the corresponding optimal dual variables. 
To eventually obtain $\nabla_{\bm{\theta}_y}J$ to improve each $\bm{\theta}^{(k)}_y$ we need to compute $\nabla_{\bm{\theta}^{(k)}}\bm x_{d,y}^{\rm{DA},*}$ and $\nabla_{\bm{\theta}^{(k)}}\bm x_{d,y}^{\rm{RT},*}$, which we discuss in the following section.
To reduce notational clutter, we omit the explicit iteration index $(k)$ moving forward unless necessary.

\subsubsection{Obtaining gradients}
\label{sssec:implicit_differentiation}
The necessary gradient components $\nabla_{\bm{\theta}}\bm x_{d,y}^{\rm{DA},*}(\bm{\theta}_y)$ and $\nabla_{\bm{\theta}}\bm x_{d,y}^{\rm{RT},*}(\bm{\theta}_y)$ can be computed via implicit differentiation.
The mechanics of this approach have been discussed previously, e.g., in \cite{amos2017optnet,besancon2023diffopt} or in \cite{degleris5169721gradient,donti2017task} with energy-specific applications.
We repeat parts of the formalism here to introduce the notation and definitions needed for the derivation of some of our results below.

First, we write the primal--dual optimality conditions of the
DA and RT problems as KKT systems using the primal feasibility constraints as written in \eqref{middlelevel} and \eqref{innerlevel} and denote the 
corresponding Lagrange multipliers by
$\bm\lambda_d^{\rm{DA}}\ge \bm 0$ and $\bm\lambda_d^{\rm{RT}}\ge \bm 0$,
respectively.

We define $\mathcal{L }_{d,y}^{\rm{DA}}$ as the DA Lagrangian:
\begin{equation}
\mathcal{L}_{d,y}^{\rm{DA}}(\bm x_{d,y}^{\rm{DA}},\bm\lambda_{d,y}^{\rm{DA}};\bm\theta_y, \bar{\bm\xi}_{d,y})
:=
C_{d,y}^{\rm{DA}}(\bm x_{d,y}^{\rm{DA}})
+\bigl(\bm\lambda_{d,y}^{\rm{DA}}\bigr)^\top\!\Bigl(
A_{d,y}^{\rm{DA}}(\bm\theta_y)\,\bm x_{d,y}^{\rm{DA}}-b_{d,y}^{\rm{DA}}(\bm\theta_y, \bar{\bm\xi}_{d,y})
\Bigr).
\end{equation}
Then, the KKT conditions for the DA problem are:
\begin{subequations}\label{eq:kkt_DA}
\begin{align}
\text{(stationarity)}\quad
&\nabla_{\bm x}\, C_{d,y}^{\rm{DA}}(\bm x_{d,y}^{\rm{DA},*})
+\bigl(A_{d,y}^{\rm{DA}}(\bm\theta_y)\bigr)^\top \bm\lambda_{d,y}^{\rm{DA},*}
=\bm 0, \label{eq:kkt_DA_stationarity}\\
\text{(primal feasibility)}\quad
&A_{d,y}^{\rm{DA}}(\bm\theta_y)\,\bm x_{d,y}^{\rm{DA},*}
- b_{d,y}^{\rm{DA}}(\bm\theta_y, \bar{\bm\xi}_{d,y})\le \bm 0, \label{eq:kkt_DA_primal}\\
\text{(dual feasibility)}\quad
&\bm\lambda_{d,y}^{\rm{DA},*}\ge \bm 0, \label{eq:kkt_DA_dual}\\
\text{(complementarity)}\quad
&\bm\lambda_{d,y}^{\rm{DA},*}\odot\Bigl(
A_{d,y}^{\rm{DA}}(\bm\theta_y)\,\bm x_{d,y}^{\rm{DA},*}
-b_{d,y}^{\rm{DA}}(\bm\theta_y, \bar{\bm\xi}_{d,y})\Bigr)=\bm 0, \label{eq:kkt_DA_comp}
\end{align}
\end{subequations}
where $\odot$ is the elementwise product.

Next, we define $\mathcal{L}_{d,y}^{\rm{RT}}$ as the RT Lagrangian:
\begin{equation}
\mathcal{L}_{d,y}^{\rm{RT}}(\bm x_{d,y}^{\rm{RT}},\bm\lambda_{d,y}^{\rm{RT}};\bm\theta_y,\bm x_{d,y}^{\rm{DA},*},\tilde{\bm\xi}_{d,y})
:=
C_{d,y}^{\rm{RT}}(\bm x_{d,y}^{\rm{RT}})
+\bigl(\bm\lambda_{d,y}^{\rm{RT}}\bigr)^\top\!\Bigl(
A_{d,y}^{\rm{RT}}(\bm\theta_y,\bm x_{d,y}^{\rm{DA},*})\,\bm x_{d,y}^{\rm{RT}}
-b_{d,y}^{\rm{RT}}(\bm\theta_y,\bm x_{d,y}^{\rm{DA},*},\tilde{\bm\xi}_{d,y})
\Bigr).
\end{equation}
Then, the KKT conditions for the RT problem are:
\begin{subequations}\label{eq:kkt_RT}
\begin{align}
\text{(stationarity)}\quad
&\nabla_{\bm x}\, C_{d,y}^{\rm{RT}}(\bm x_{d,y}^{\rm{RT},*})
+\bigl(A_{d,y}^{\rm{RT}}(\bm\theta_y,\bm x_{d,y}^{\rm{DA},*})\bigr)^\top
\bm\lambda_{d,y}^{\rm{RT},*}
=\bm 0, \label{eq:kkt_RT_stationarity}\\
\text{(primal feasibility)}\quad
&A_{d,y}^{\rm{RT}}(\bm\theta_y,\bm x_{d,y}^{\rm{DA},*})\,\bm x_{d,y}^{\rm{RT},*}
- b_{d,y}^{\rm{RT}}(\bm\theta_y,\bm x_{d,y}^{\rm{DA},*},\tilde{\bm\xi}_{d,y})\le \bm0,
\label{eq:kkt_RT_primal}\\
\text{(dual feasibility)}\quad
&\bm\lambda_{d,y}^{\rm{RT},*}\ge \bm 0, \label{eq:kkt_RT_dual}\\
\text{(complementarity)}\quad
&\bm\lambda_{d,y}^{\rm{RT},*}\odot\Bigl(
A_{d,y}^{\rm{RT}}(\bm\theta_y,\bm x_{d,y}^{\rm{DA},*})\,\bm x_{d,y}^{\rm{RT},*}
-b_{d,y}^{\rm{RT}}(\bm\theta_y,\bm x_{d,y}^{\rm{DA},*},\tilde{\bm\xi}_{d,y})
\Bigr)=\bm 0. \label{eq:kkt_RT_comp}
\end{align}
\end{subequations}
With slight abuse of notation, we write the KKT conditions in \eqref{eq:kkt_DA} and \eqref{eq:kkt_RT} as KKT operators $\kappa_{d,y}^{\rm{DA}}$ and $\kappa_{d,y}^{\rm{RT}}$:
\begin{align}
\kappa_{d,y}^{\rm{DA}}(\bm z_{d,y}^{\rm{DA}},\bm\theta_y, \bar{\bm\xi}_{d,y})&=\bm 0,
\quad \bm z_{d,y}^{\rm{DA}}:=\bigl(\bm x_{d,y}^{\rm{DA}},\bm\lambda_{d,y}^{\rm{DA}}\bigr),
\label{eq:da_kkt_op}\\
\kappa_{d,y}^{\rm{RT}}(\bm z_{d,y}^{\rm{RT}},\bm\theta_y,\bm x_{d,y}^{\rm{DA},*},\tilde{\bm\xi}_{d,y})&=\bm 0,
\quad \bm z_{d,y}^{\rm{RT}}:=\bigl(\bm x_{d,y}^{\rm{RT}},\bm\lambda_{d,y}^{\rm{RT}}\bigr),
\label{eq:rt_kkt_op}
\end{align}
where each operator collects the corresponding stationarity, primal feasibility,
dual feasibility, and complementarity conditions rewritten into equality-constrained form. Under standard regularity conditions (in particular local invertibility of the Jacobian of the KKT operators) the \textit{implicit function theorem} \citep{dontchev2009implicit} suggests that the implicit solution maps from parameters $\bm{\theta}_y$ to the optimal primal--dual solutions $z_{d,y}^{\rm DA,*}$, $z_{d,y}^{\rm RT,*}$ are differentiable and their derivatives can be expressed as \citep{amos2017optnet}:
\begin{equation}
\frac{\partial \bm z_{d,y}^{\rm{DA},*}}{\partial \bm\theta_y}
=
-\Bigl[\partial_{\bm z}\kappa_{d,y}^{\rm{DA}}\!\bigl(\bm z_{d,y}^{\rm{DA},*}(\bm\theta_y),\bm\theta_y, \bar{\bm\xi}_{d,y}\bigr)\Bigr]^{-1}
\,\partial_{\bm\theta_y}\kappa_{d,y}^{\rm{DA}}\!\bigl(\bm z_{d,y}^{\rm{DA},*}(\bm\theta_y),\bm\theta_y, \bar{\bm\xi}_{d,y}\bigr),
\label{eq:ift_DA}
\end{equation}
\begin{equation}
\frac{\partial \bm z_{d,y}^{\rm{RT},*}}{\partial \bm\theta_y}
=
-\Bigl[\partial_{\bm z}\kappa_{d,y}^{\rm{RT}}\!\bigl(\bm z_{d,y}^{\rm{RT},*}(\bm\theta_y),\bm\theta_y,\bm x_{d,y}^{\rm{DA},*}(\bm\theta_y),\tilde{\bm\xi}_{d,y}\bigr)\Bigr]^{-1}
\,\partial_{\bm\theta_y}\kappa_{d,y}^{\rm{RT}}\!\bigl(\bm z_{d,y}^{\rm{RT},*}(\bm\theta_y),\bm\theta_y,\bm x_{d,y}^{\rm{DA},*}(\bm\theta_y),\tilde{\bm\xi}_{d,y}\bigr).
\label{eq:ift_RT}
\end{equation}

Notably, the RT problem has a direct dependence on the investment parameters $\bm \theta_y$ and an indirect dependence on the optimal DA dispatch $\bm x^{DA,*}_{d,y}$. 
We resolve the resulting dependencies in \eqref{eq:ift_RT} via the chain rule:
\newcommand{\virtheight}{\vphantom{\left.\frac{\partial \bm x_{d,y}^{\rm{RT},*}}{\partial \bm\theta_y}\right|_{\bm x_{d,y}^{\rm{DA},*}}}}
\begin{equation}
\frac{\partial \bm x_{d,y}^{\rm{RT},*}}{\partial \bm\theta_y}
=
\underbrace{
\left.\frac{\partial \bm x_{d,y}^{\rm{RT},*}}{\partial \bm\theta_y}\right|_{\bm x_{d,y}^{\rm{DA},*}}}_{\bm J_{d,y}^{\rm{RT},\theta_y}}
+
\underbrace{\virtheight
\frac{\partial \bm x_{d,y}^{\rm{RT},*}}{\partial \bm x_{d,y}^{\rm{DA},*}}}_{\bm J_{d,y}^{\rm{RT},x}}
\cdot
\underbrace{\virtheight
\frac{\partial \bm x_{d,y}^{\rm{DA},*}}{\partial \bm\theta_y}}_{\bm J_{d,y}^{\rm{DA},\theta_y}}.
\label{eq:chain_rule_rt}
\end{equation}
We then extract the blocks corresponding to the primal variables from $\frac{\partial \bm z_{d,y}^{\rm{DA},*}}{\partial \bm\theta_y}$ and
$\frac{\partial \bm z_{d,y}^{\rm{RT},*}}{\partial \bm\theta_y}$ to form $\bm J_{d,y}^{\rm{DA},\theta_y}$ and the ``direct'' term $\bm J_{d,y}^{\rm{RT},\theta_y}$ in \eqref{eq:chain_rule_rt}.
To obtain the ``indirect'' Jacobian $\bm J_{d,y}^{\rm{RT},x}=\frac{\partial \bm x_{d,y}^{\rm{RT},*}}{\partial \bm x_{d,y}^{\rm{DA},*}}$, we apply the same implicit function theorem argument to the RT KKT system while treating $\bm x_{d,y}^{\rm{DA},*}$ as an explicit parameter:
\begin{equation}
\frac{\partial \bm z_{d,y}^{\rm{RT},*}}{\partial \bm x_{d,y}^{\rm{DA},*}}
=
-\Bigl[\partial_{\bm z}\kappa_{d,y}^{\rm{RT}}\!\bigl(\bm z_{d,y}^{\rm{RT},*}(\bm\theta_y),\bm\theta_y,\bm x_{d,y}^{\rm{DA},*}(\bm\theta_y),\tilde{\bm\xi}_{d,y}\bigr)\Bigr]^{-1}
\,\partial_{\bm x^{\rm{DA}}}\kappa_{d,y}^{\rm{RT}}\!\bigl(\bm z_{d,y}^{\rm{RT},*}(\bm\theta_y),\bm\theta_y,\bm x_{d,y}^{\rm{DA},*}(\bm\theta_y),\tilde{\bm\xi}_{d,y}\bigr),
\label{eq:ift_RT_xDA}
\end{equation}
and we again take the primal block of \eqref{eq:ift_RT_xDA} to obtain $\bm J_{d,y}^{\rm{RT},x}$.

We can now form the gradient of $J(\bm\theta)$ by combining the
investment cost derivatives with the implicit differentiation of the DA
and RT operating costs. We write
$\bm{\theta}= (\bm\theta_1,\dots,\bm\theta_Y)$ and define the feasible
investment set as
$\mathcal H=\mathcal H_1\times\cdots\times\mathcal H_Y$. The full
gradient is the stacked vector
$
\nabla_{\bm\theta} J(\bm\theta)
=
\big(
\nabla_{\bm\theta_1}J(\bm\theta_1),
\dots,
\nabla_{\bm\theta_Y}J(\bm\theta_Y)
\big),
$
where each year-$y$ block is
\begin{align}
\nabla_{\bm\theta_y}J(\bm\theta_y)
=
\left(\frac{1}{1+r}\right)^{\tau(y)}
\Bigg[
\nabla_{\bm\theta_y} C_y^{\mathrm{inv}}(\bm\theta_y)
+
\frac{365}{D}
\sum_{d \in [D]}
\Big[
\bigl(\bm J_{d,y}^{\rm{RT},\theta_y}\bigr)^{\!\top}
\nabla_{\bm x_{d,y}^{\rm{RT}}}
\bm C_{d,y}^{\rm{RT}}
\big(
\bm x_{d,y}^{\rm{RT},*}
(\bm\theta_y,\bm x_{d,y}^{\rm{DA},*}(\bm\theta_y))
\big)
\nonumber\\
+
\bigl(\bm J_{d,y}^{\rm{DA},\theta_y}\bigr)^{\!\top}
\Big(
\nabla_{\bm x_{d,y}^{\rm{DA}}}
\bm C_{d,y}^{\rm{DA}}
\big(
\bm x_{d,y}^{\rm{DA},*}(\bm\theta_y)
\big)
+
\bigl(\bm J_{d,y}^{\rm{RT},x}\bigr)^{\!\top}
\nabla_{\bm x_{d,y}^{\rm{RT}}}
\bm C_{d,y}^{\rm{RT}}
\big(
\bm x_{d,y}^{\rm{RT},*}
(\bm\theta_y,\bm x_{d,y}^{\rm{DA},*}(\bm\theta_y))
\big)
\Big)
\Big]
\Bigg].
\label{eq:tri_level_gradient_block}
\end{align}

\subsubsection{Projected gradient step}
\label{sssec:gradient_descent_algortihm}

We use the the gradients $\nabla_{\bm\theta_y^{(k)}} J\big(\bm\theta_y^{(k)}\big)$ to solve problem~\eqref{tri-level} by iteratively updating the investment vector $\theta^{(k)}$.
We collect the stochastic gradients for all years $y$ in $\bm{g}^{(k)}$ (see also Step~8 in Algorithm~\ref{alg:sgd_tri-level} below) and define a stepsize $\alpha^{(k)} >0$.
Because an unconstrained gradient step
$\bm\theta^{(k)} - \alpha^{(k)} \bm g^{(k)}$ does not generally satisfy the investment feasibility constraints
\eqref{limit_inv}--\eqref{monocity_inv}, we apply an additional projection step onto the feasible investment set. 
We define the projection operator as
 $ \Pi_{\mathcal H}(\bm v) := \arg\min_{\bm u \in \mathcal H} \|\bm u - \bm v\|_2^2.$
Then the projected gradient step is:
\begin{equation}
\bm\theta^{(k+1)}
=
\arg\min_{\bm u \in \mathcal H}
\|\bm u - (\bm\theta^{(k)} - \alpha^{(k)} \bm g^{(k)})\|_2^2.
\label{eq:projected_update}
\end{equation}
The projection operations solve a strongly convex quadratic program that returns the feasible point closest (in Euclidean distance) to the unconstrained gradient step.
This projection enforces the required investment bounds and intertemporal monotonicity at every iteration while preserving as much of the descent direction as possible. 
See also \cite{beck2017first}.
Fig.~\ref{fig:tri_level_scheme} summarizes the entire process.
Notably, the projection step can only enforce constraints on $\bm{\theta}$. The following section now addresses dealing with the renewable target constraint \eqref{outer_renewable_target}.

\subsubsection{Computational complexity}

The practical computation of the needed gradients, i.e., $
\nabla_{\bm\theta} J(\bm\theta)$, is efficient. 
After solving the DA and RT problems, instead of explicitly computing the entire Jacobian $\frac{\partial \bm{x}^*_{d,y}}{\partial \bm{\theta_y}}$, we can directly evaluate the adjoint product $\left(\frac{\partial \bm{x}^*_{d,y}}{\partial \bm{\theta}_y}\right)^{\!\top}\bm{v}$, with
$\bm{v}=\frac{\partial J_y}{\partial \bm{x}^*_{d,y}}$.
This can be computed efficiently using the available primal--dual optimal solution and numerical auto-differentiation \citep{amos2017optnet,besancon2023diffopt}.

Overall, at each iteration $(k)$, for each year $y \in [Y]$, estimating the year-$y$ gradient block $\nabla_{\theta_y} J(\bm\theta_y)$ from a batch of $B$ sampled scenarios requires $B$ solves of the DA model and $B$ solves of the RT model.
Computing the gradient for each sampled scenario and DA/RT solution
requires solving a linear system of similar scale to the original optimization problem.
As a result, a batch gradient evaluation requires solving $2YB$ linear systems associated with the DA and RT KKT operators.
As a result, the problem is solved by repeating a computational procedure of manageable size, even for larger-scale operations problems, which also lends itself to parallelization.

In contrast, traditional MPEC reformulations of bilevel problems that embed all KKT conditions and complementarity constraints into a single optimization problem, quickly become intractable as the the problem size increases exponentially.
Consequently, although each batch iteration introduces additional linear-system solves for sensitivity analysis, the overall computational burden in the proposed approach grows in a controlled and scalable manner relative to the single-level equivalent formulation \citep{rosemberg2025general,besancon2023diffopt}.

\begin{figure}
\centering
\resizebox{0.9\columnwidth}{!}{%
\begin{tikzpicture}[scale=0.65,transform shape,
    node distance=0.9cm,
    box/.style={
        draw=#1!70!black,
        thick,
        rounded corners,
        align=center,
        fill=#1!12,
        inner sep=3pt,
        text width=3.1cm,
        minimum height=0.85cm
    },
    box/.default=gray,
    arr/.style={->, thick, draw=#1!75!black},
    arr/.default=gray,
    lbl/.style={font=\scriptsize, text=black!70},
    txt/.style={font=\scriptsize},
    >=Latex
]

\node[box=blue, txt] (inv) {%
\textbf{Outer Investment planning level}\\[0.1em]
$\bm\theta_y=(\bar{\bm p}_y,\bar{\bm f}_y,\bar{\bm p}_y^{w},\bar{\bm p}_y^{B},\bar{\bm E}_y)$\\
$
\bm\theta
=
(\bm\theta_1,\dots,\bm\theta_Y)
$
};

\node[box=purple, txt, below=of inv] (da) {%
\textbf{Inner DA operational level}\\[0.2em]
 $\forall d \in [D], \forall y \in [Y]$ :\\[0.2em]
$\bm x_{d,y}^{\mathrm{DA},*}(\bm\theta_y)$
};

\node[box=purple, txt, below=of da] (rt) {%
\textbf{Inner RT operational level}\\[0.2em]
 $\forall d \in [D], \forall y \in [Y]$:\\[0.2em]
$\bm x_{d,y}^{\mathrm{RT},*}\!\big(\bm\theta_y,\bm x_{d,y}^{\mathrm{DA},*}(\bm\theta_y)\big)$
};

\node[box=teal, txt, right=2.2cm of rt] (grad) {%
\textbf{Differentiable component}\\[0.2em]
Jacobians + chain rule\\[0.2em]
$\bm g^{(k)}$
};

\node[box=green, txt, above=1.05cm of grad] (proj) {%
\textbf{Projection step}\\[0.2em]
$\bm\theta^{(k+1)}=\Pi_{\mathcal H}\!\left(\bm\theta^{(k)}-\alpha^{(k)}\bm g^{(k)})\right)$
};

\node[draw=gray!90, dashed, rounded corners, align=center, txt, left=1.2cm of da] (forecast) {$\bm \bar{\xi}_{d,y}$ \\$ \forall d\in [D], \forall y \in [Y] $\\Renewable injection\\ forecast};
\node[draw=gray!90, dashed, rounded corners, align=center, txt, left=1.2cm of rt] (real) {$\tilde{\bm \xi}_{d,y}$ \\$ \forall d\in [D], \forall y \in [Y] $\\Renewable injection \\realization};

\draw[arr=blue]   (inv) -- node[right, lbl]{parameters $\bm \theta_y$} (da);
\draw[arr=orange] (da)  -- node[right, lbl]{DA schedule $\bm x^{\text{DA},*}_{d,y}$} (rt);

\draw[arr] (forecast.east) -- (da.west);

\draw[arr] (real.east) -- (rt.west);

\draw[arr=purple] (rt.east) -- node[above, lbl]{$\bm J_{d,y}^{\mathrm{RT},\theta_y}$, $\bm J_{d,y}^{\mathrm{RT},\bm x}$} (grad.west);

\draw[arr=orange]
  (da.east) to[bend right=25]
  node[midway, above, sloped, lbl]{$\bm J_{d,y}^{\mathrm{DA},\theta_y}$}
  (grad.north west);

\draw[arr=teal]   (grad.north) -- node[right, lbl]{gradient step} (proj.south);
\draw[arr=green]  (proj.north) |- node[right, lbl]{projected $\bm \theta^{(k+1)}$} (inv.east);

\end{tikzpicture}%
}
\caption{Schematic representation of our proposed investment--operation modeling framework and the projected SGD algorithm used to solve it. 
}
\label{fig:tri_level_scheme}
\end{figure}
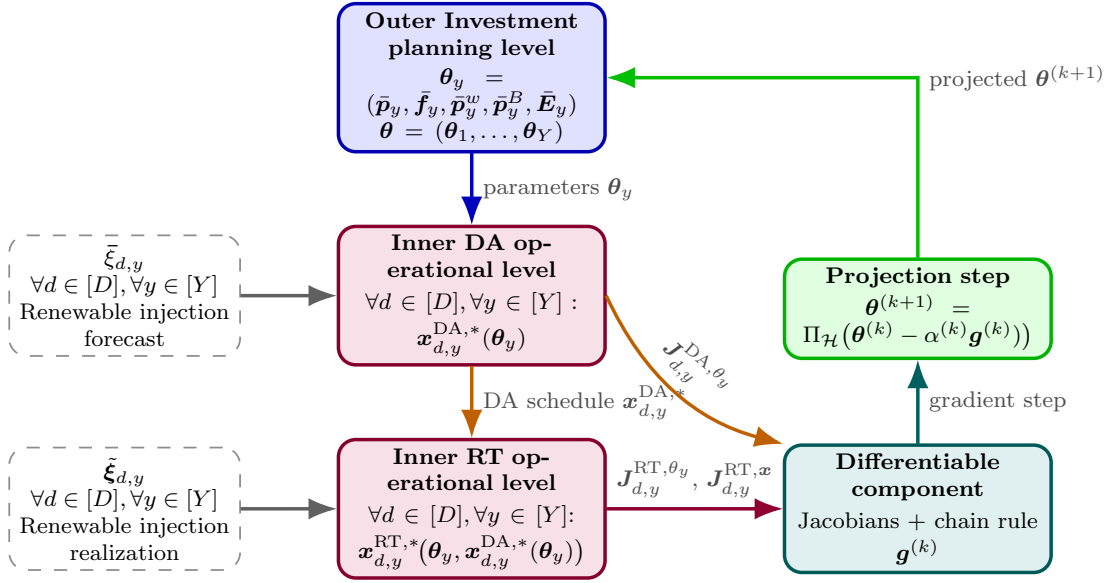

\subsection{Renewable-constrained planning objective with dynamic penalties}

The renewable-energy target \eqref{outer_renewable_target} depends on the operational outcomes of both the DA and RT stages and can therefore only be implicitly enforced through the investment decisions $\bm{\theta}$.
To remain compatible with the proposed gradient-based solution method, we enforce the renewable requirement through a dynamic penalty formulation.
We define the in-batch renewable shortfall, i.e., the total amount of renewable production that is \textit{missing} to meet the target renewable production share of $\eta_y$, over the days from batch $\mathcal{B}_{y}$ as:
\begin{equation}
\label{ren_violation}
\nu_{y,\mathcal{B}_y}(\bm \theta_y)
=
 \frac{365}{B} \sum_{d \in \mathcal{B}_y} \sum_{t \in [T]} \Bigg[ \eta_y \Bigg( \sum_{g \in [G]}
\left(
p_{g,d,t,y}
+
p^+_{g,d,t,y}
-
p^-_{g,d,t,y}
\right)
+
\sum_{w \in [W]} p^{w,RT}_{w,d,t,y}
\Bigg)
-
\sum_{w \in [W]} p^{w,RT}_{w,d,t,y}
\Bigg],
\end{equation}
Notably, the renewable target is achieved when $\nu_y\le0$. 

To internalize information on this constraint into the gradient update steps of $\bm{\theta}$ we augment the problem objective $\bm{J}(\theta)$ to the Lagrangian:
\begin{equation}
\mathcal{L}(\bm \theta, \bm \lambda)
=
J(\bm \theta)
+
\sum_{y \in [Y]} \lambda_y \nu_{y,\mathcal{B}_y}(\bm \theta_y),
\label{outer_obj}
\end{equation}
where $\bm \lambda$ denotes the vector of Lagrange multipliers associated with the renewable constraints.
At each iteration $(k)$, these multipliers are updated via dual ascent, i.e., we have different $\lambda_y^{(k)}$ for each iteration additional guide the $\bm{\theta}$ update towards meeting this constraint. See also \cite{wang2023learning}.

Because the renewable target shortfall $\nu_{{y},\mathcal{B}_y}$ is sensitive to the specific properties of the scenarios in batch $\mathcal{B}_y$, 
updating $\bm{\lambda}$ based on a single-batch $\nu_{{y},\mathcal{B}_y}$ would produce unstable behavior. 
For example, one iteration may sample a set of low-wind days, producing a large shortfall and sharply increasing the penalty multiplier. A subsequent iteration may contain samples from windy spring periods with excess renewable generation, causing the penalty to collapse. 
To overcome this and stabilize the dual updates, we use a momentum-based variance reduction technique by applying a smoothing exponential moving average (EMA) across the previous iterations to the renewable-target shortfall \citep{cutkosky2019momentum,cui2025two}:
\begin{equation}
\label{ema_update}
V_y^{(k)}
=
(1-\beta)V_y^{(k-1)}
+
\beta  \nu_{{y},\mathcal B_{y}^{(k)}}(\bm \theta_y^{(k)}),
\end{equation}
where $\beta \in (0,1)$ is the smoothing parameter.
The smoothed constraint violation allows the algorithm to react to persistent characteristics of all scenarios rather than batch-specific properties. 
Using the smoothed violation estimate, we update the dual multiplier through a projected dual ascent step:
\begin{equation}
\label{dual_ascent_update}
\lambda_y^{(k+1)}
\leftarrow
\min\left\{
\bar\lambda,
\max\left(0,\lambda_y^{(k)}+\alpha_\lambda V_y^{(k)}\right)
\right\},
\end{equation}
where $ \bar{\bm \lambda}$ is a pre-defined upper bound and $\alpha_{\lambda}$ denotes the dual stepsize.
If $V_y^{(k)} > 0$, the system is persistently violating the renewable target, causing the multiplier $\lambda_y^{(k)}$ to increase. In subsequent iterations, the renewable penalty term in the outer objective becomes more expensive, incentivizing additional renewable investment capacity to restore compliance. Conversely, if $V_y^{(k)} < 0$, the system is over-satisfying the renewable requirement, causing the multiplier $\lambda_y^{(k)}$ to reduce.
This means that persistent constraint fulfillment is tracked within the dual update step.
However, because the renewable target constraint is an inequality constraint, we do not include information of this constraint into the $\bm{\theta}$ update step if $\nu_{y,\mathcal B_{y}^{(k)}}(\bm \theta_y^{(k)})\le0$. See also Step~7 in Algorithm~\ref{alg:sgd_tri-level} where we define ``effective'' multipliers $ \omega_y = \mathbb{1}_{\{V_y>0\}} \lambda_y$, with $\mathbb{1}_{\{\cdot\}}$ being the indicator function.


\begin{algorithm}
\caption{Projected SGD for Uncertainty-Aware Power System Planning}
\label{alg:sgd_tri-level}
\DontPrintSemicolon
\SetKwInOut{Input}{Input}
\SetKwInOut{Output}{Output}

\Input{%
  Feasible initial investment vectors $\bm \theta_y^{(0)} \in \mathcal{H}_y \;\forall y \in [Y]$; \\
  initial dual multipliers $\bm \lambda^{(0)} = \bm \lambda_0$; initial EMA $V_y^{(0)} = 0 \;\forall y \in [Y]$; \\
  max. number of outer iterations $K$; batch size $B$, with
$|\mathcal B_{y}^{(k)}|=B$ for all $k$ and $y$; \\
  gradient descent stepsize  $\alpha $; dual learning rate $\alpha_\lambda$; EMA smoothing factor $\beta$; \\
  penalty upper bound $\bar\lambda$; set of sample days $[D] = \{1,\dots,D\}$.
}
\Output{%
Optimal investment plan
  $\bm \theta^* = (\bm \theta_1^*,\dots,\bm \theta_Y^*)$.
}\;

\For{$k = 0,1,\dots,K-1$}{
\ForEach{$y \in [Y]$}{
  1. Sample scenario batch
   $\mathcal B_{y}^{(k)} \subseteq [D]$ uniformly at random.\;
   
  2. \ForEach{$d \in \mathcal B_{y}^{(k)}$}{
    Solve DA and RT problem sequence with current $\bm \theta_y^{(k)}$,
    and obtain optimal operational decisions $\bm x_{d,y}^{\mathrm{DA},*}(\bm \theta_y^{(k)})$ and
    $\bm x_{d,y}^{\mathrm{RT},*}(\bm \theta_y^{(k)}, \bm x_{d,y}^{\mathrm{DA},*}(\bm \theta_y^{(k)}))$\;
  }

    3. Evaluate batch renewable shortfall $\nu_{y,\mathcal B_{y}^{(k)}}(\bm \theta_y^{(k)})$ using \eqref{ren_violation} \;
    4. Update EMA renewable shortfall:
    $
    V_y^{(k)} \leftarrow
    \begin{cases}
    \nu_{y,\mathcal B_{y}^{(k)}}(\bm \theta_y^{(k)}), & k=0,\\
    (1-\beta)V_y^{(k-1)}+\beta \nu_{y,\mathcal B_{y}^{(k)}}(\bm \theta_y^{(k)}), & k>0,
    \end{cases}
    $ \;}

  5. Compute the planning objective-gradient blocks:
 $
  \bm g_{J}^{(k)}
  \gets
  \big(
  \bm g_{J,1}^{(k)},
  \dots,
  \bm g_{J,Y}^{(k)}
  \big)
  $
  where
  $
    \bm g_{J,y}^{(k)} \; \gets\; \Big(\frac{1}{1+r}\Big)^{\tau(y)} \Bigg[ \nabla_{\bm \theta_y} C^{\mathrm{inv}}_y(\bm \theta_y^{(k)})
    + \frac{365}{B}\sum_{d \in \mathcal B_{y}^{(k)}}
      \nabla_{\bm\theta_y} C^{\mathrm{DA+RT}}_{d,y}(\bm \theta_y^{(k)}) \Bigg] 
      $\;
6. Compute renewable constraint-gradient blocks:
$
\bm g_{\nu}^{(k)}
\gets
\big(
\bm g_{\nu_1}^{(k)},
\dots,
\bm g_{\nu_Y}^{(k)}
\big)
$
where
$
\bm g_{\nu_y}^{(k)}
\gets
\nabla_{\bm \theta_y}
\big[
\nu_{y,\mathcal B_{y}^{(k)}}(\bm \theta_y^{(k)})
\big],
\qquad
\forall y\in[Y].
$\;
7. Define effective penalty multipliers:
$
\bm\omega^{(k)}
\gets
\big(
\omega_1^{(k)},
\dots,
\omega_Y^{(k)}
\big)
$
where
$
\omega_y^{(k)}
\gets
\mathbb{1}_{\{V_y^{(k)}>0\}}
\lambda_y^{(k)},
\qquad
\forall y\in[Y].
$
\;
8. Compute the full primal descent direction:
$
\bm g^{(k)}
\gets
\bm g_{J}^{(k)}
+
\big(
\omega_1^{(k)}\bm g_{\nu_1}^{(k)},
\dots,
\omega_Y^{(k)}\bm g_{\nu_Y}^{(k)}
\big)
$
\;

9. Projected gradient step:
$
\bm\theta^{(k+1)}
\leftarrow
\Pi_{\mathcal H}
\big(\bm\theta^{(k)}-\alpha\bm g^{(k)}\big).
$


10. Bounded penalty update:
$
\lambda_y^{(k+1)}
\leftarrow
\min\left\{
\bar\lambda,
\max\left(0,\lambda_y^{(k)}+\alpha_\lambda V_y^{(k)}\right)
\right\},
\qquad
\forall y\in[Y].
$
\;
11. Stopping criterion: \\
\If{$\bm\theta^{(k)}$ \text{satisfies approximate stationarity (see Theorem~\ref{theorem1})} \textbf{or} $k=K-1$}{
\textbf{break};
}
}
Set $\bm \theta^*
\leftarrow
\bm \theta^{(K)}$ \;
\end{algorithm}

 \subsection{Convergence analysis}
 \label{sec:Convergence Analysis}
This section provides some theoretical insights on the convergence of the presented gradient descent approach.
Our analysis builds on the results by \cite{degleris5169721gradient}, but extends them to the multi-year, multi-stage operational structure with renewable targets considered in this work. 
Since Algorithm~\ref{alg:sgd_tri-level} employs a projected stochastic gradient method with batching and an EMA-stabilized penalty update, we establish convergence in expectation of the projected stationarity residuals associated with the augmented objective \eqref{outer_obj}.
In particular, we show that the iterates converge to a first-order stationary region whose size is determined by the variance of the stochastic gradient estimator.

We first introduce and discuss some necessary regularity assumptions.
Recall that for each year
$\mathcal H_y \subset \mathbb R^{n}$
denotes the feasible investment set defined by
\eqref{limit_inv}--\eqref{monocity_inv}.

\begin{assumption}[Compact Feasible Set]
For every year $y\in[Y]$, the feasible set $\mathcal H_y$ is nonempty, convex, and compact.
\label{assumption1}
\end{assumption}
This assumption mainly requires investments to be continuous. 
This is arguably the strongest assumption that we make in this paper, but that is common in practical planning models, e.g., \cite{osti_1788425}. 
A relaxation of this assumption may be possible by invoking additional techniques, e.g., along the lines of \cite{boland2018combining}, but we consider this subject to further research and beyond the scope of this paper.

\begin{assumption}[Feasibility]
For every $\bm \theta_y \in \mathcal H_y$, every day $d$, and every year
$y\in[Y]$, the DA problem \eqref{middlelevel} and RT problem
\eqref{innerlevel} are feasible.
\label{assumption2}
\end{assumption}

\begin{assumption}[Uniqueness]
For every $\bm \theta_y \in \mathcal H_y$, every day $d$, and every year $y \in [Y]$, the DA problem \eqref{middlelevel} admits a unique optimal solution 
$\bm x_{d,y}^{\rm{DA},*}(\bm \theta_y)$, and the RT problem \eqref{innerlevel} admits a unique optimal solution 
$\bm x_{d,y}^{\rm{RT},*}(\bm \theta_y, \bm x_{d,y}^{\rm{DA},*}(\bm \theta_y))$.
\label{assumption3}
\end{assumption}
Assumptions~\ref{assumption2} and \ref{assumption3} are non-critical and easily met by many practical market-clearing formulations, which typically model generator production costs as strictly convex quadratic functions, which
promotes the uniqueness of the primal optimal solution. 
When strict convexity is not guaranteed (e.g., under linear production costs), we can add a small quadratic
regularization term (e.g., $\gamma \|\bm x\|_2^2$ with $\gamma>0$) to ensure a unique and differentiable
solution map. 
To ensure feasibility, we can include load-shedding and renewable-curtailment variables with high cost coefficients.

\begin{assumption}[Regularity of DA and RT problems]
\label{assumption4}
For every $\bm\theta_y \in \mathcal H_y$, every day $d$, and every year $y\in[Y]$ the following holds:
\begin{enumerate}[label=(\roman*)]
    \item  At the optimal solution of the DA problem \eqref{middlelevel}, the gradients of the active constraints are linearly independent (i.e., the linear independence constraint qualification (LICQ) holds), and the problem is nondegenerate,
    \item At the optimal solution of the RT problem \eqref{innerlevel}, the gradients of the active constraints are linearly independent, and the problem is nondegenerate.
\end{enumerate}
\end{assumption}
Assumption~\ref{assumption4} imposes standard regularity conditions for
constrained optimization problems. The linear independence constraint
qualification (LICQ) requires that, at the optimal solution, the
gradients of all active inequality constraints are linearly independent.
In other words, no active constraint gradient can be expressed as a
linear combination of the others. The strict complementarity condition
requires that every active inequality constraint has a strictly positive
associated Lagrange multiplier. This condition prevents degeneracies in
the KKT system and rules out cases in which multiple multiplier vectors
satisfy the optimality conditions.
In our proposed expansion-planning framework,
Assumption~\ref{assumption4} is generally expected to hold under normal
operating conditions. The DA and RT problems are convex optimization
problems with linear network and operational constraints. LICQ is
satisfied whenever the active constraints at the optimum do not contain
redundant binding relations, which is typically the case in practical
power system models. Likewise, strict complementarity fails only in
exceptional situations where an active constraint has a zero multiplier
or where multiple active constraints become exactly dependent. Such
situations are uncommon in practice and usually arise only under highly
specific operating conditions. Therefore,
Assumption~\ref{assumption4} represents a standard regularity condition
that is typically satisfied in practical linearized optimal power flow problems.
Together with the uniqueness assumption
(Assumption~\ref{assumption3}), these regularity conditions ensure that
the KKT system is locally well posed and that its Jacobian with respect
to the primal--dual variables is nonsingular at the solution.
Consequently, the implicit function theorem guarantees local
differentiability of the DA and RT solution maps with respect to the
investment parameters $\bm\theta_y$, as discussed in
Section~\ref{sssec:implicit_differentiation} above.

\begin{assumption}[Unbiased Estimators and Bounded Variance]
\label{assumption5}
At iteration $(k)$, let $\bm g_{J}^{(k)}$ denote the stochastic gradient
estimator of the objective function $J$, and let
\[
\bm g_{\nu}^{(k)}
=
\big(
\bm g_{\nu_1}^{(k)},
\dots,
\bm g_{\nu_Y}^{(k)}
\big)
\]
denote the stacked stochastic gradient estimator of the renewable
constraint violations. We define
\[
\bm\nu(\bm\theta^{(k)})
=
\big(
\nu_1(\bm\theta_1^{(k)}),
\dots,
\nu_Y(\bm\theta_Y^{(k)})
\big),
\qquad
\nabla_{\bm\theta}\bm\nu(\bm\theta^{(k)})
=
\big(
\nabla_{\bm\theta_1}\nu_1(\bm\theta_1^{(k)}),
\dots,
\nabla_{\bm\theta_Y}\nu_Y(\bm\theta_Y^{(k)})
\big).
\]
Let $\nu_{{y},\mathcal B_{y}^{(k)}}^{(k)}$ denote the stochastic value of the renewable
constraint violation for year $y$, computed from the batch
$\mathcal B_{y}^{(k)}$. The estimators satisfy
\[
\mathbb E
\big[
\bm g_{J}^{(k)}
\mid
\bm\theta^{(k)}
\big]
=\!
\nabla_{\bm\theta}
J(\bm\theta^{(k)}),
\!\!\quad
\mathbb E
\big[
\bm g_{\nu}^{(k)}
\mid
\bm\theta^{(k)}
\big]
=\!
\nabla_{\bm\theta}
\bm\nu(\bm\theta^{(k)}),
\!\!\quad
\mathbb E
\big[
\nu_{y,\mathcal B_{y}^{(k)}}(\bm \theta_y^{(k)})
\mid
\bm\theta_y^{(k)}
\big]
=\!
\nu_y(\bm\theta_y^{(k)})
\quad
\forall y\!\in\![Y].
\]
Moreover, the estimators have uniformly bounded variances:
\[
\begin{aligned}
&
\mathbb E
\Big[
\|
\bm g_{J}^{(k)}
-
\nabla_{\bm\theta}
J(\bm\theta^{(k)})
\|^2
\mid
\bm\theta^{(k)}
\Big]
\le
\sigma_J^2,
\\
&
\mathbb E
\Big[
\|
\bm g_{\nu}^{(k)}
-
\nabla_{\bm\theta}
\bm\nu(\bm\theta^{(k)})
\|^2
\mid
\bm\theta^{(k)}
\Big]
\le
\sigma_{\nabla\nu}^2,
\\
&
\mathbb E
\Big[
\|
\nu_{y,\mathcal B_{y}^{(k)}}(\bm \theta_y^{(k)})
-
\nu_y(\bm\theta_y^{(k)})
\|^2
\mid
\bm\theta_y^{(k)}
\Big]
\le
\sigma_\nu^2,
\qquad
\forall y\in[Y],
\end{aligned}
\]
for some constants
$\sigma_J^2,\sigma_{\nabla\nu}^2,\sigma_\nu^2<\infty$.
\end{assumption}
This assumption holds in our setting because we construct the stochastic
estimators by sampling batches
$\mathcal B_{y}^{(k)} \subset [D]$
and averaging the corresponding day-level operational outcomes for each
year $y\in[Y]$. 
Since we sample the days uniformly, the resulting
estimators remain unbiased with respect to the full objective gradient
$\nabla_{\bm\theta}J(\bm\theta)$, the expected renewable shortfall
$\nu_y(\bm\theta_y)$, and its gradient
$\nabla_{\bm\theta_y}\nu_y(\bm\theta_y)$.
In addition, Algorithm~\ref{alg:sgd_tri-level} projects the dual
multipliers onto the bounded interval $[0,\bar\lambda]$, which ensures
uniform boundedness of the effective multipliers
$\omega_y^{(k)}$.

For the following discussion,
we write the effective augmented objective at iteration $(k)$ as
\begin{align}
\Phi(\bm\theta^{(k)})
=
J(\bm\theta^{(k)})
+
\sum_{y\in[Y]}
\omega_y^{(k)}
\nu_y(\bm\theta_y^{(k)}), 
\end{align}
and then define the following projected-gradient residuals: 
\begin{align}
&\widehat G^{(k)}
=
\left(
\bm\theta^{(k)}
-
\Pi_{\mathcal H}
\left(
\bm\theta^{(k)}
-
\alpha\bm g^{(k)}
\right)
\right) \label{eq:proj-stat-residual}\\
& G^{(k)}
=
\left(
\bm\theta^{(k)}
-
\Pi_{\mathcal H}
\left(
\bm\theta^{(k)}
-
\alpha
\nabla_{\bm\theta}\Phi(\bm\theta^{(k)})
\right)
\right).
\end{align}
The first residuals $\widehat G^{(k)}$ are the stochastic projected-gradient residuals, and the second residuals $G^{(k)}$ are the deterministic projected-gradient residuals.

\begin{lemma}\label{lemma: variance bound} Suppose Assumptions~\ref{assumption1}--\ref{assumption5} hold. We let $\mathcal F^{(k)}$ denote the history of
Algorithm~\ref{alg:sgd_tri-level}, including all iterates and sampled batches up to iteration $(k-1)$ \citep{ebrahimi2025resolution}, and we define $\sigma_g^2=2\sigma_J^2+2\bar\lambda^2\sigma_{\nabla\nu}^2$.
Then, the gradient estimator has a bounded second order moment:
\[
\mathbb E
\left[
\left\|
\bm g^{(k)}
-
\nabla_{\bm\theta}\Phi(\bm\theta^{(k)})
\right\|^2
\mid
\mathcal F^{(k)}
\right]
\le
\sigma_g^2,
\]
\end{lemma}
\begin{proof}
    See Appendix~\ref{ax:proof_of_thrm1}.
\end{proof}

\begin{theorem}[Projected SGD convergence with deterministic stationarity residuals]
\label{theorem1}
Suppose Assumptions~\ref{assumption1}--\ref{assumption5} hold. Assume that each $\Phi$ has $L_\Phi$-Lipschitz continuous gradients on
$\mathcal H$, and that $\Phi$ is uniformly bounded from below on
$\mathcal H$ by $\Phi_*$.
\begin{enumerate}[label=(\alph*)]
    \item  \emph{[Non-asymptotic error bounds]} For any $K\ge1$, $\alpha\le \frac{1}{2L_\Phi}$, and $\sigma_g^2
=
2\sigma_J^2
+
2\bar\lambda^2\sigma_{\nabla\nu}^2$, we have
\begin{align*}
\frac{1}{K}
\sum_{k=0}^{K-1}
\mathbb E
\left[
\left\|
G^{(k)}
\right\|^2
\right]
&\le
\frac{4}{L_\Phi K}
\left(
\Phi(\bm\theta^{(0)})-\Phi_*
\right)
+
\frac{2\alpha\sigma_g^2}{L_{\Phi}}
+
{2\alpha^2\sigma_g^2},
\end{align*}
Consequently, there exists an iteration $k^*\in\{0,\dots,K-1\}$ such that
\begin{align*}
\mathbb E
\left[
\left\|
G^{(k^*)}
\right\|^2
\right]
&\le
\frac{4}{L_\Phi K}
\left(
\Phi(\bm\theta^{(0)})-\Phi_*
\right)
+
\frac{2\alpha\sigma_g^2}{L_{\Phi}}
+
{2\alpha^2\sigma_g^2}.
\end{align*}

\item \emph{[Complexity bound]} Let $\epsilon>0$ be an arbitrary scalar such that $\mathbb E
\left[
\left\|
G_{k_\epsilon^*}
\right\|^2
\right]\le \epsilon$. Suppose the stepsize in Algorithm~\ref{alg:sgd_tri-level} is given by $\frac{1}{k_\epsilon^*}$. Then, we have
\[k_\epsilon^*=\mathcal{O}\left(\frac{1}{L_\Phi \epsilon}
\left(
\Phi(\bm\theta^{(0)})-\Phi_*
\right)+\frac{\sigma_g^2}{\epsilon L_{\Phi}}+\frac{\sigma_g^2}{\sqrt{\epsilon}}\right).
\]
\end{enumerate}
\end{theorem}
\begin{proof}
    See Appendix~\ref{ax:proof_of_thrm1}.
\end{proof}

Theorem~\ref{theorem1} establishes convergence properties of the proposed
projected stochastic primal-dual framework. The theorem shows that the
average projected stationarity residual satisfies an
$\mathcal O(1/K)$ convergence bound up to stochastic error terms
determined by the gradient-estimation variance and the chosen stepsize.
Consequently, the iterates generated by
Algorithm~\ref{alg:sgd_tri-level} approaches a first-order stationary
region of the constrained expansion-planning problem. Part~(b) further
provides an iteration-complexity bound for obtaining an
$\epsilon$-approximate stationary solution. Moreover, if full-batch
gradients are used, or if the batch size increases so that
$\sigma_g^2\rightarrow 0$, the stochastic error terms vanish and the
result reduces to the deterministic projected-gradient convergence rate.

\section{Numerical experiments}

We demonstrate the proposed method on an exemplary power system planning study and compare the proposed method with a standard benchmark approach.
We first give an overview on the data and simulation
setup, and then discuss the results.

\subsection{Data and implementation}
We use the IEEE 24-bus test system \citep{ordoudis2016updated} for the analysis. 
We extend the available load data from the original IEEE 24-bus test system to one year of hourly data by scaling it with yearly real-world demand profiles from ENTSO-E \citep{OpenPowerSystemDataplatform}. 
For this case study, we model demand as deterministic and introduce operational uncertainty $\bm \xi$ explicitly only for wind power injections. We note that this is not a requirement for the proposed method but a choice for this case study to ease exposition and discussion.
To this end, we use data from the NLR (previously NREL) Wind Integration National Dataset (WIND) Toolkit \citep{draxl2015wind}. The dataset provides high-resolution wind data for locations across the United States, including day-ahead forecasts and real-time measurements of wind power generation.
In addition to the resources from the data set as described in \cite{ordoudis2016updated}, we place wind farms at nodes 3, 5, 9, 16, 19, 20 and the capacity of each wind farm is 400 MW.
We set renewable curtailment and load shedding cost as $\$20/MWh$, and $\$10, 000 /MWh$, respectively.
The planner can make resource investments as per the candidate wind farms, batteries, fuel generators and transmission lines shown in Tables~\ref{wind} --~\ref{line}.

\begin{table}
\centering
\caption{Candidate wind generation data}
\begin{tabular*}{.9\linewidth}{@{\extracolsep{\fill}} llll @{}}
\midrule
Wind Gen & Bus & $w_{\max}$ (MW) &  $c_{\text{inv}}$ (\$/MW) \\
\hline
Wind Gen 1 & 3  & 1800  & $90000 $ \\
Wind Gen 2 & 5 & 1500  & $80000 $ \\
Wind Gen 3 & 9 & 2000  & $85000 $ \\
Wind Gen 4 & 16  & 1500  & $90000 $ \\
Wind Gen 5 & 19 & 2000  & $100000 $ \\
Wind Gen 6 & 20 & 1400  & $85000 $ \\
\bottomrule
\label{wind}
\end{tabular*}
\end{table}

\begin{table}
\centering
\caption{Battery data}
\begin{tabular*}{0.9\linewidth}{@{\extracolsep{\fill}} lllll @{}}
\toprule
Battery & Bus & $E$ (MWh) &   $c_{\text{inv}}$ (\$/MWh)& $c_{\text{opr}}$ (\$/MWh)\\
\hline
Bat 1 & 16 & 640 &   $120000 $ & $1.5 $ \\
Bat 2 & 20 & 380  &  $100000 $ & $2.5 $ \\
Bat 3  & 7  & 160 &  $110000 $ & $4 $ \\
Bat 4 & 24 & 200 &  $170000 $ & $5 $ \\
\bottomrule
\label{battery}
\end{tabular*}
\end{table}

\begin{table}
\centering
\caption{Fuel generation data}
\begin{tabular*}{0.9\linewidth}{@{\extracolsep{\fill}} lllll @{}}
\toprule
Fuel Gen & Bus & $p_{\max}$ (MW)& $c_{\text{inv}} $ (\$/MW)& $c_{\text{opr}}$ (\$/MWh) \\
\hline
Fuel Gen 1 & 8  & 350.0  &  $40000 $ & $42.0 $ \\
Fuel Gen 2 & 21 & 280.0  &  $42000 $ & $12.0 $ \\
Fuel Gen 3 & 9  & 310.0 &  $25000 $ & $54.0 $ \\
Fuel Gen 4 & 17 & 412.0 &  $25000 $ & $26.0 $ \\
Fuel Gen 5 & 19 & 315.0 &  $25000 $ & $17.0 $ \\
Fuel Gen 6 & 20 & 220.0 &  $34000 $ & $82.0 $ \\
Fuel Gen 7 & 16 & 475.0  &  $35000 $ & $42.0 $ \\
\bottomrule
\label{generation}
\end{tabular*}
\end{table}

\begin{table}
\centering
\caption{Candidate transmission line data}
\begin{tabular*}{0.9\linewidth}{@{\extracolsep{\fill}} llllll @{}}
\toprule
Line & From bus & To bus & $x_{\text{pu}}$ & Rate (MW) &  $c_{\text{inv}}$ (\$/MW)\\
\hline
Line 1 & 3  & 9  & 0.12 & 200.0  & $100.0 $ \\
Line 2 & 16 & 19 & 0.10 & 250.0 & $90.5 $ \\
\bottomrule
\label{line}
\end{tabular*}
\end{table}

We implemented the model and algorithm in Julia using JuMP \citep{lubin2023jump}. To compute the Jacobians
$\frac{\partial \bm{x}^*_d}{\partial \bm{\theta}}$ we use DiffOpt \citep{besancon2023diffopt} in \emph{reverse mode} \citep{rosemberg2024two}. 
The resulting optimization problems were solved with the Gurobi solver on the Rutgers Amarel cluster, using nodes with 128~GB of memory and 16 cores (Dual Intel Xeon Gold 6448Y processors). For the single-year setting, the average runtime per sample (day) was approximately 11 seconds for each iteration of the proposed SGD. Our data and implementation is available open source at 
\begin{center}
[\url{https://github.com/ropes-lab/Uncertainty-aware-PSP-via-SGD}]. 
\end{center}

\subsection{Testing and benchmark}
\label{ssec:testing_system}

To evaluate the performance of the investment decisions obtained with the proposed uncertainty-aware method, we compare them with investments obtained using a deterministic benchmark method that, as common in most planning models, uses a single-stage operations model with \textit{realized} renewable availability. 
This benchmark model, which we will refer to as \textit{uncertainty agnostic}, enforces the same constraints on investment decisions and operations (power balance, power flow, etc.) as the proposed model. However, because the investment objective is only cost minimization and operations are modeled as a single stage without uncertainty, the resulting model can be cast as a single linear program and solved directly with off-the shelf solvers. We refer to Appendix~\ref{ax:benchmark_model} for the exact model formulation. 

Once investment decisions are obtained, we perform out-of-sample tests that simulate the two-stage operations with imperfect forecasts using these decisions. The out-of-sample data is drawn from the same wind data set, but uses new days that have not been used in the models to obtain either the proposed uncertainty-aware investment decisions or the uncertainty-agnostic benchmark. 
The DA and RT simulation stages correspond to solving the DA and RT problems as described in \eqref{daprob} and \eqref{RTprob}.

\subsection{Results and discussion}
We first analyze the behavior of the algorithm and then compare its results with the discussed benchmark.

\subsubsection{Convergence}

We tuned the hyperparameters empirically by running the model with different parameter values and stepsize schedules. 
The final configuration achieved stable convergence and consistent performance across the tested scenarios.
At each iteration, we sample five days, i.e., $B = 5$, and we use a step size $\alpha = 0.0004$. 
In addition, we use an EMA smoothing factor of $\beta = 0.2$ to track the renewable shortfall trends, while we update the adaptive penalty parameter using a dual stepsize of $\alpha_{\lambda} = 0.5$. 
We initialize the adaptive penalty parameter as 
$\lambda^{(0)} = 5\,\bar c_{\mathrm{wind}}$
in Algorithm~\ref{alg:sgd_tri-level}, where 
$\bar c_{\mathrm{wind}}$ denotes the average of the fixed investment costs over all candidate wind farms. 
We solve problem \eqref{tri-level} under four distinct renewable energy target levels: $\eta_y \in \{30\%, 40\%, 50\%, 60\%\}$.

Figs.~\ref{fig:convergenceplots}a and \ref{fig:convergenceplots}b show the trajectory of the augmented objective $\mathcal{L}(\bm{\theta}, \bm{\lambda})$ as per \eqref{outer_obj} and the original cost objective $\bm{J}(\bm{\theta})$, respectively, during the solution process.
Fig. \ref{fig:convergenceplots}c shows the renewable shortfall in each iteration estimated per batch (light blue line) and with the applied EMA smoothing (orange line).
In the first few iterations, i.e., when investments $\bm{\theta}$ are close to their initial value, costs are low (Fig.~\ref{fig:convergenceplots}b), but the renewable target is violated (Fig.~\ref{fig:convergenceplots}b).
As a result, the Lagrange multiplier increases and strongly affects the augmented objective. 
When the algorithm is still far from convergence, these dual corrections dominate the pure cost and cause the observed positive and negative swings. 
As the algorithm progresses, the renewable violation becomes systematically non-positive, indicating that the target is satisfied or exceeded. 
In this case, the algorithm does not apply an additional penalty, and the dual variable of the renewable target is driven back to zero. 
At convergence, the renewable constraint is satisfied without applying additional penalties, and the remaining objective primarily reflects the objective of the cost-minimizing feasible solution. (Observe that the trajectories in Figs.~\ref{fig:convergenceplots}a and \ref{fig:convergenceplots}b converge to the same value.)

Fig.~\ref{fig:convergenceplots}d shows the projected stationarity residual $G^{(k)}$ as per \eqref{eq:proj-stat-residual} and its moving average.  
The residual decreases rapidly during the initial iterations and continues to decline more gradually thereafter, indicating that the required gradients approach a first-order stationary region. 
The moving average smooths the stochastic fluctuations induced by batch sampling and reveals a clear downward trend.
As shown in
Theorem~\ref{theorem1}, the residual converges to a small positive value.

\begin{figure}[t]
    \centering
    \includegraphics[width=0.9\linewidth]{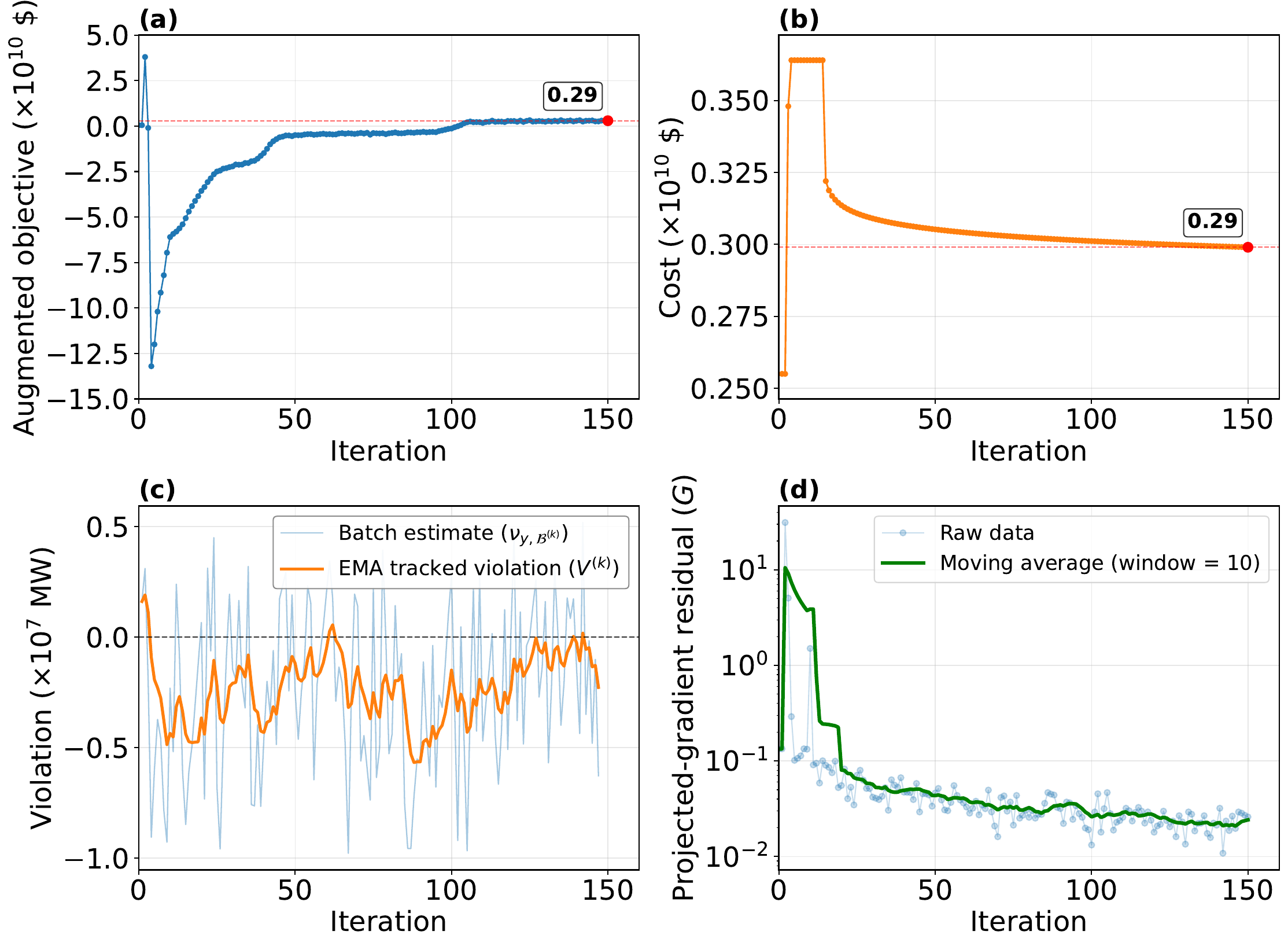}
    \caption{(a): Augmented objective values for projected SGD with a Primal-Dual method ($B=5$) under the renewable target $\eta_y = 0.5$. (b):  Objective values without penalty term for projected SGD with a Primal-Dual method ($B=5$) under the renewable target $\eta_y = 0.5$. (c): The expected renewable energy shortfall $\nu_{y,\mathcal{B}_y^{(k)}}(\bm\theta_y^{(k)})$ across iterations. (d): Trajectory of the projected-gradient residual $G^{(k)}$ as per \eqref{eq:proj-stat-residual} and its moving average (window size 10).}
    \label{fig:convergenceplots}
\end{figure}

\subsubsection{Investment decisions}

We now analyze and discuss the resulting investment decisions and highlight the role of accounting for forecast uncertainty in the planning model. 
Fig.~\ref{fig:ost30-90} shows the investment decisions obtained for the four renewable target levels. 
Each row corresponds to a different renewable target, while the columns compare investments in battery energy capacity, transmission expansion, fuel-based generation, and wind generation. 
The figure contrasts the solutions produced by the proposed uncertainty-aware approach with those obtained from the uncertainty-agnostic benchmark.
Across all renewable targets, the uncertainty-aware investments consistently prioritize more flexible resources, particularly in battery storage and transmission capacity. 
As the renewable target increases from $30\%$ to $60\%$, the uncertainty-aware solution significantly expands battery capacity,
reflecting the increased need for temporal flexibility to manage renewable variability and uncertainty. 
In contrast, the uncertainty-agnostic benchmark maintains relatively modest battery investments.
Since the benchmark assumes full information about future realizations, it anticipates reduced flexibility needs.

The uncertainty-agnostic investment model systematically under-invests in transmission capacity. 
The differences in transmission expansion are particularly significant at higher renewable targets. While both methods invest similarly in candidate line 1, the uncertainty-aware method substantially increases expansion on line 2 for renewable targets $50\%$ and $60\%$. 

Fuel-based generation investments remain relatively similar between the two approaches across all renewable targets. However, the uncertainty-agnostic benchmark generally opts for slightly higher capacities for several candidate generators, i.e., relying more on dispatchable generation resources that can be scheduled efficiently with complete foresight. 
The uncertainty-aware model, on the other hand, puts slightly less emphasis on additional fuel generation investments in favor of flexibility through battery storage and transmission.

Interestingly, the largest structural differences appear in wind generation investments. The proposed uncertainty-aware approach spreads renewable capacity more broadly across multiple wind sites and progressively increases wind investments as the renewable target rises. 
At higher targets, especially $\eta_y = 60\%$, the proposed approach allocates substantial capacity to wind farms 3, 5, and 1, creating a more geographically diversified renewable portfolio that mitigates the risk of localized low-wind conditions. 
In contrast, the uncertainty-agnostic benchmark concentrates investments more selectively at a few wind farms, particularly wind farms 3, 4, and 6.

Overall, the Fig.~\ref{fig:ost30-90} illustrates that uncertainty-aware investment planning systematically favors flexibility and diversification to hedge against uncertainty and that an uncertainty-agnostic approach leads to less diversified investment strategies.

\begin{figure}[h]
    \centering    \includegraphics[width=0.9\linewidth]{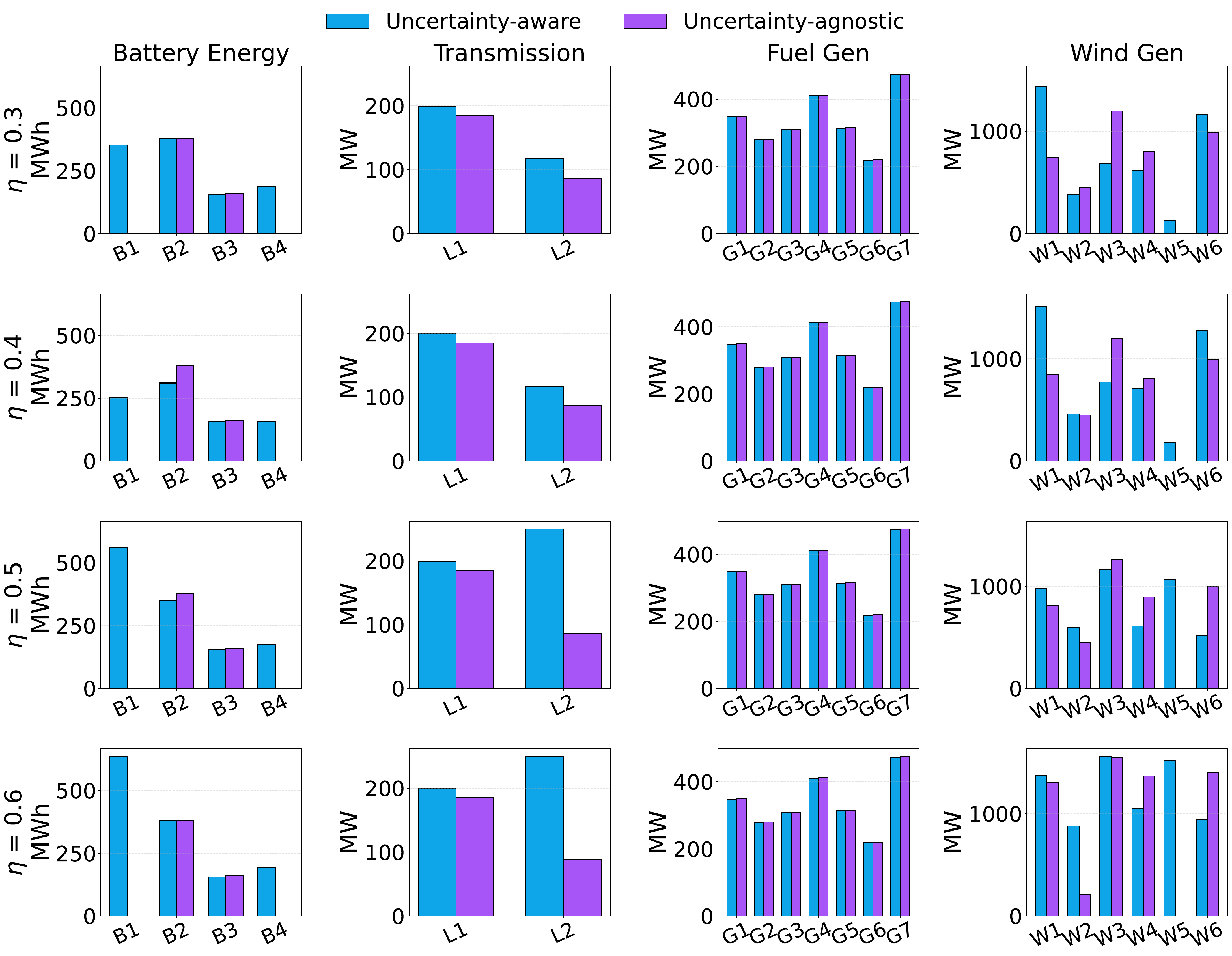}
    \caption{Comparison of optimal expansion decisions from the uncertainty-aware method (blue) and the uncertainty-agnostic benchmark (purple).}
    \label{fig:ost30-90}
\end{figure}

\subsubsection{Cost}

Table~\ref{tab:expansion} shows the out-of-sample cost comparison between the uncertainty-agnostic benchmark and the proposed uncertainty-aware approach when we evaluate the investment decisions by simulating them over one year of two-stage DA and RT operations using new data. 
The reported values include investment cost, operating cost, and total system cost.
The numbers shown in parentheses represent the percentage difference of the uncertainty-aware solution relative to the uncertainty-agnostic model for each cost category.
In addition, we also report the share of renewable production that was realized during operations. We highlight again, that the DA and RT operations models only pursue a cost-minimal dispatch and do not enforce renewable targets.

First, our results indicate that the proposed uncertainty-aware approach consistently achieves lower total costs while also satisfying the renewable targets. 
Although the resulting investment costs are higher compared to the uncertainty-agnostic investments, the gain in system flexibility significantly reduces operating costs in the presence of forecast uncertainty. 
For example, at the 50\% renewable target, the uncertainty aware model requires a 29.69\% more expensive investment but saves 9.24\% in operations, leading to an overall 4.07\%
reduction in total cost.
This indicates that higher upfront investments allow for substantial operational savings under forecast uncertainty.
The largest investment increase occurs at the 60\% target, where the uncertainty-aware investments cost 39.51\% more than the uncertainty-agnostic investments.
As a result, both approaches show a similar performance in cost, but only the uncertainty-aware approach manages to 
achieve the planned 60\% renewable target.
With the uncertainty-agnostic investments only 52\% renewable share is realized.
Despite the significant increase in investment cost, the uncertainty-aware investments still slightly reduce total cost by \(0.09\%\), mainly due to lower operation cost.

Our results indicate that neglecting forecast errors in planning decisions may lead to underinvestment in flexible resources. 
Although the uncertainty-agnostic model appears to reduce investment costs, the resulting system experiences higher operating costs and fails to fully achieve the renewable penetration targets, particularly at high target levels. 
In contrast, internalizing forecast uncertainty into the investment planning decisions leads to an improved return-on-investment during operations due to
improved renewable integration and more economical system performance.

\begin{table}[htbp]
\centering
\caption{Impact of forecast errors on expansion planning. Numbers in parentheses indicate the percentage difference of the (a) uncertainty-aware  solution relative to the (b) uncertainty-agnostic  model.}
\label{tab:expansion}
\resizebox{\textwidth}{!}{%
\begin{tabular}{llcccc}
\hline
 & & \multicolumn{4}{c}{Renewable target $\eta$} \\
\cline{3-6}
Model & Cost (m\$)
& 30\% & 40\% & 50\% & 60\% \\
\hline

(a) & Inv. cost
& 471.69 (+20.02\%) 
& 487.38 (+21.54\%) 
& 533.07 (+29.69\%) 
& 714.41 (+39.51\%) \\
& Op. cost 
& 2,630.39 (-6.72\%) 
& 2,610.54 (-6.91\%) 
& 2,436.30 (-9.24\%) 
& 2,378.56 (-7.94\%) \\
& Total (Inv. + Op.)
& 3,102.08 (-3.45\%)  
& 3,097.93 (-3.35\%)  
& 2,969.38 (-4.07\%)  
& 3,092.97 (-0.09\%) \\
& Realized renewable
& 30\% & 40\% & 50\% & 60\% \\\\

(b) & Inv. cost 
& 393.02 & 401.02 & 411.04 & 512.09 \\
& Op. cost 
& 2,820.02 & 2,804.18 & 2,684.32 & 2,583.68 \\
& Total (Inv. + Op.)
& 3,213.04 & 3,205.20  & 3,095.36 & 3,095.78 \\
& Realized renewable
& 30\% & 39\% & 45\% & 52\% \\
\hline
\end{tabular}
}
\end{table}

\section{Conclusion}
Motivated by the growing magnitude of operational uncertainty in power system planning, this paper formulated a power system expansion planning framework that explicitly accounts for coupled day-ahead (DA) and real-time (RT) operations under imperfect forecasts. We formulated the planning problem as a multi-level optimization model that allows us to (i) explicitly model the independence of day-to-day system operations from planning decisions, (ii) account for forecast uncertainty between the DA and RT stages, (iii) preserve time-coupling operational constraints, and (iv) enforce renewable production targets.

To solve the proposed model we described
a projected stochastic gradient descent (SGD) method that leverages implicit differentiation of the operational optimization problems to efficiently compute planning gradients. 
To enforce renewable production targets, we introduced a dynamic primal--dual formulation that combines projected dual updates with a smoothing mechanism to stabilize the renewable shortfall estimates obtained from stochastic sampling. 
We also established differentiability and convergence results that provide theoretical guarantees for the proposed algorithm.

Our numerical study demonstrated the application of the proposed uncertainty-aware framework and compared the resulting investment decisions with those obtained from a standard uncertainty-agnostic benchmark model that assumes perfect operational foresight. 
Our results showed that explicitly accounting for operational uncertainty improves the ability of the system to satisfy renewable utilization requirements due to a higher prioritization of flexible battery storage and transmission investments during the planning stage.
Across a set of multiple renewable targets, the proposed approach achieved lower total system costs while meeting the desired renewable utilization targets.

Further research is needed for options to more accurately model discrete (``lumpy'') investment decisions. 
A potential path forward here is the use of an inner convex relaxation of the resulting non-convex investment space along the lines of the methods discussed in \cite{blondel2022efficient}.
Also, extensive studies on larger-scale, practical planning models are needed to confirm the scalability results from \cite{degleris5169721gradient} with our proposed approach.
Based on the analysis in \cite{degleris5169721gradient}, we are confident about the scalability of the method presented here, but we note that \cite{degleris5169721gradient} uses a custom implementation of the KKT system and the gradient computation. Further research is needed to confirm that the more general and automatic implicit differentiation via packages like DiffOpt (the one used for this paper), can compete here.

\appendix

\section{Convergence analysis}
\label{app:convergence}
We establish the main results underlying the
convergence analysis of Algorithm~\ref{alg:sgd_tri-level}. In
Lemma~\ref{lem:diff_maps}, we show that the optimal DA and RT operating
decisions are continuously differentiable functions of the investment
variables and that their Jacobians are bounded. In
Lemma~\ref{lem:lipschitz_grad}, we use these properties to establish
Lipschitz continuity of the gradients of the objective function and the
renewable violation measures. We then derive a bound on the stochastic
gradient estimation error in Lemma~\ref{lemma: variance bound}. Finally,
we use these results to prove the convergence result stated in
Theorem~\ref{theorem1}.

Before establishing the convergence result, we characterize the
gradient of the penalized objective and the gradient of the renewable
violation measures with respect to the investment variables.
At iteration $(k)$, the exact gradient of $\Phi$ with respect to the
stacked investment vector $\bm\theta$ is
\[
\nabla_{\bm\theta}\Phi(\bm\theta^{(k)})
=
\nabla_{\bm\theta}J(\bm\theta^{(k)})
+
\big(
\omega_1^{(k)}
\nabla_{\bm\theta_1}\nu_1(\bm\theta_1^{(k)}),
\dots,
\omega_Y^{(k)}
\nabla_{\bm\theta_Y}\nu_Y(\bm\theta_Y^{(k)})
\big)
.\]
Similarly, we obtain the gradient of the expected renewable constraint
violation for year $y$ via the chain rule:
\begin{align}
\nabla_{\bm\theta_y}\nu_y(\bm\theta_y)
&=
\sum_{d=1}^{D}
\nabla_{\bm\theta_y}\nu_{y,d}(\bm\theta_y).
\label{eq:gradNu_app}
\end{align}
Where we obtain $\nabla_{\bm\theta_y}\nu_{y,d}(\bm\theta_y)$ as follows
\begin{align}
\nabla_{\bm\theta_y}\nu_{y,d}(\bm\theta_y)
&=
\frac{\partial \nu_{y,d}}
{\partial \bm x_{d,y}^{\rm{DA},*}}
\frac{\partial \bm x_{d,y}^{\rm{DA},*}}
{\partial \bm\theta_y}
+
\frac{\partial \nu_{y,d}}
{\partial \bm x_{d,y}^{\rm{RT},*}}
\left(
\frac{\partial \bm x_{d,y}^{\rm{RT},*}}
{\partial \bm\theta_y}
+
\frac{\partial \bm x_{d,y}^{\rm{RT},*}}
{\partial \bm x_{d,y}^{\rm{DA},*}}
\frac{\partial \bm x_{d,y}^{\rm{DA},*}}
{\partial \bm\theta_y}
\right).
\label{eq:gradNu_chain_app}
\end{align}

\begin{lemma}[Differentiability of operational solution maps]
\label{lem:diff_maps}
Under Assumptions~\ref{assumption1}--\ref{assumption4}, for every
$y\in[Y]$ and $d\in[D]$, the solution maps
\[
\bm\theta_y
\mapsto
\bm x_{d,y}^{\rm{DA},*}(\bm\theta_y),\qquad
(\bm\theta_y,\bm x_{d,y}^{\rm{DA}})
\mapsto
\bm x_{d,y}^{\rm{RT},*}
\left(
\bm\theta_y,
\bm x_{d,y}^{\rm{DA}}
\right)
\]
are continuously differentiable on $\mathcal H_y$. Moreover, their
Jacobians are bounded on $\mathcal H_y$.
\end{lemma}

\begin{proof}
Assumption~\ref{assumption3} gives unique optimal DA and RT solutions.
Assumption~\ref{assumption4} gives LICQ and nondegeneracy at the DA and RT
solutions. These conditions imply that the Jacobian of each KKT system
with respect to the corresponding primal--dual variables is nonsingular at
the solution. The implicit function theorem then gives local continuous
differentiability of the DA and RT solution maps with respect to their
parameters. Since $\mathcal H_y$ is compact by
Assumption~\ref{assumption1}, continuity of the derivatives implies
bounded Jacobians on $\mathcal H_y$.
\end{proof}

\begin{lemma}[Lipschitz continuity]
\label{lem:lipschitz_grad}
Under Assumptions~\ref{assumption1}--\ref{assumption4},
$\nabla_{\bm\theta}J(\bm\theta)$ and
$\nabla_{\bm\theta_y}\nu_y(\bm\theta_y)$, $y\in[Y]$, are Lipschitz
continuous on their feasible domains.
\end{lemma}

\begin{proof}
Assumption~\ref{assumption1} makes each $\mathcal H_y$ compact. The DA and
RT feasible regions remain bounded over the feasible investment region
under Assumption~\ref{assumption2} and the operational constraints.
Lemma~\ref{lem:diff_maps} gives bounded Jacobians for the DA and RT
solution maps.
Since the investment cost, operating cost, and renewable
violation functions have continuous derivatives on the relevant compact
domains, their derivatives are bounded. Therefore,
$\nabla_{\bm\theta}J$ and $\nabla_{\bm\theta_y}\nu_y$ are Lipschitz
continuous. Because the effective multipliers satisfy
$0 \le \omega_y^{(k)} \le \bar\lambda$,
the penalized objective $\Phi$ also has a uniformly Lipschitz
continuous gradient with constant $L_\Phi$ independent of $(k)$.
\end{proof}

\subsection{Proofs for Lemma 1 and Theorem 1}
\label{ax:proof_of_thrm1}

Our proofs follow the projected gradient argument of \cite{beck2017first} and the projected gradient convergence analysis used in \cite{degleris5169721gradient}. 
The main difference is that our update uses stochastic gradients, and it requires an additional variance bound and expectation arguments to control the gradient-estimation error. In the following, we prove Lemma~\ref{lemma: variance bound} and Theorem~\ref{theorem1}, respectively.

\begin{proof}[Proof of Lemma~\ref{lemma: variance bound}]At iteration $(k)$, the stochastic primal direction is $\bm g^{(k)}=\bm g_{J}^{(k)}+\bm\omega^{(k)}\odot \bm g_{\nu}^{(k)}$. The corresponding exact gradient of $\Phi$ is
\[
\nabla_{\bm\theta}\Phi(\bm\theta^{(k)})
=
\nabla_{\bm\theta}J(\bm\theta^{(k)})
+
\bm\omega^{(k)}
\odot
\nabla_{\bm\theta}\bm\nu(\bm\theta^{(k)}).
\]
Thus, we obtain
\[
\begin{aligned}
\bm g^{(k)}
-
\nabla_{\bm\theta}\Phi(\bm\theta^{(k)})
=
\bm g_{J}^{(k)}
-
\nabla_{\bm\theta}J(\bm\theta^{(k)})
+
\bm\omega^{(k)}
\odot
\left(
\bm g_{\nu}^{(k)}
-
\nabla_{\bm\theta}\bm\nu(\bm\theta^{(k)})
\right).
\end{aligned}
\]
Using
$\|\bm a+\bm b\|^2\le 2\|\bm a\|^2+2\|\bm b\|^2$ and the preceding relation, we obtain
\[
\begin{aligned}
&
\left\|
\bm g^{(k)}
-
\nabla_{\bm\theta}\Phi(\bm\theta^{(k)})
\right\|^2
\le
2
\left\|
\bm g_{J}^{(k)}
-
\nabla_{\bm\theta}J(\bm\theta^{(k)})
\right\|^2
+
2
\left\|
\bm\omega^{(k)}
\odot
\left(
\bm g_{\nu}^{(k)}
-
\nabla_{\bm\theta}\bm\nu(\bm\theta^{(k)})
\right)
\right\|^2 .
\end{aligned}
\]
Because $0\le \omega_y^{(k)}\le \bar\lambda$ for all $y\in[Y]$, we have
\[
\left\|
\bm\omega^{(k)}
\odot
\left(
\bm g_{\nu}^{(k)}
-
\nabla_{\bm\theta}\bm\nu(\bm\theta^{(k)})
\right)
\right\|^2
\le
\bar\lambda^2
\left\|
\bm g_{\nu}^{(k)}
-
\nabla_{\bm\theta}\bm\nu(\bm\theta^{(k)})
\right\|^2 .
\]
Taking the conditional expectation and using Assumption~\ref{assumption5}, we
obtain
\[
\mathbb E
\left[
\left\|
\bm g^{(k)}
-
\nabla_{\bm\theta}\Phi(\bm\theta^{(k)})
\right\|^2
\mid
\mathcal F^{(k)}
\right]
\le
\sigma_g^2,
\]
where $\sigma_g^2
=
2\sigma_J^2
+
2\bar\lambda^2\sigma_{\nabla\nu}^2$.
\end{proof}

\begin{proof}[Proof of Theorem~\ref{theorem1}]

\begin{enumerate}[label=(\alph*)]
    \item 
Since $\nabla \Phi$ is $L_\Phi$-Lipschitz continuous, the descent lemma yields
\begin{align}
\Phi(\bm\theta^{(k+1)})
&\le
\Phi(\bm\theta^{(k)})
+
\nabla_{\bm\theta}\Phi(\bm\theta^{(k)})^\top
\left(
\bm\theta^{(k+1)}
-
\bm\theta^{(k)}
\right)
+
\frac{L_\Phi}{2}
\left\|
\bm\theta^{(k+1)}
-
\bm\theta^{(k)}
\right\|^2 .\label{equation: descent lemma}
\end{align}
For the second term on the right-hand side, adding and subtracting $\bm g^{(k)}$, we obtain
\[
\begin{aligned}
\nabla_{\bm\theta}\Phi(\bm\theta^{(k)})^\top
\left(
\bm\theta^{(k+1)}
-
\bm\theta^{(k)}
\right)
&=
\bm g^{(k)\top}
\left(
\bm\theta^{(k+1)}
-
\bm\theta^{(k)}
\right)
-
\left(
\bm g^{(k)}
-
\nabla_{\bm\theta}\Phi(\bm\theta^{(k)})
\right)^\top
\left(
\bm\theta^{(k+1)}
-
\bm\theta^{(k)}
\right).
\end{aligned}
\]
Substituting the preceding relation into~\eqref{equation: descent lemma}, we obtain
\begin{align}
\label{eq:second_descent}
\Phi(\bm\theta^{(k+1)})
&\le
\Phi(\bm\theta^{(k)})
\!+\!
\bm g^{(k)\top}
\left(
\bm\theta^{(k+1)}
\!-\!
\bm\theta^{(k)}
\right) -
\left(
\bm g^{(k)}
\!-\!
\nabla_{\bm\theta}\Phi(\bm\theta^{(k)})
\right)^\top
\left(
\bm\theta^{(k+1)}
\!-\!
\bm\theta^{(k)}
\right)
\!+\!\nonumber\\&
\frac{L_\Phi}{2}
\left\|
\bm\theta^{(k+1)}
\!-\!
\bm\theta^{(k)}
\right\|^2.
\end{align}
The projection optimality condition for $\bm\theta^{(k+1)}
=
\Pi_{\mathcal H}
\left(
\bm\theta^{(k)}
-
\alpha \bm g^{(k)}
\right)$
gives
\[
\left(
\bm\theta^{(k)}
-
\alpha\bm g^{(k)}
-
\bm\theta^{(k+1)}
\right)^\top
\left(
\bm\theta^{(k)}
-
\bm\theta^{(k+1)}
\right)
\le 0.
\]
By rearranging the terms, we obtain
\[
\bm g^{(k)\top}
\left(
\bm\theta^{(k+1)}
-
\bm\theta^{(k)}
\right)
\le
-
\frac{1}{\alpha}
\left\|
\bm\theta^{(k+1)}
-
\bm\theta^{(k)}
\right\|^2 .
\]
Substituting the preceding inequality into~\eqref{eq:second_descent}, we obtain
\[
\begin{aligned}
\Phi(\bm\theta^{(k+1)})
&\le
\Phi(\bm\theta^{(k)})
-
\left(
\frac{1}{\alpha}
-
\frac{L_\Phi}{2}
\right)
\left\|
\bm\theta^{(k+1)}
-
\bm\theta^{(k)}
\right\|^2
-
\left(
\bm g^{(k)}
-
\nabla_{\bm\theta}\Phi(\bm\theta^{(k)})
\right)^\top
\left(
\bm\theta^{(k+1)}
-
\bm\theta^{(k)}
\right).
\end{aligned}
\]
Using Young's inequality $-\bm a^\top \bm b
\le
\frac{1}{2\alpha}\|\bm b\|^2
+
\frac{\alpha}{2}\|\bm a\|^2$, with $\bm a=
\bm g^{(k)}
-
\nabla_{\bm\theta}\Phi(\bm\theta^{(k)})$ and $\bm b=
\bm\theta^{(k+1)}
-
\bm\theta^{(k)}$, we obtain
\[
\begin{aligned}
\Phi(\bm\theta^{(k+1)})
&\le
\Phi(\bm\theta^{(k)})
-
\left(
\frac{1}{2\alpha}
-
\frac{L_\Phi}{2}
\right)
\left\|
\bm\theta^{(k+1)}
-
\bm\theta^{(k)}
\right\|^2
+
\frac{\alpha}{2}
\left\|
\bm g^{(k)}
-
\nabla_{\bm\theta}\Phi(\bm\theta^{(k)})
\right\|^2 .
\end{aligned}
\]
Since $\alpha\le 1/(2L_\Phi)$, we have $\frac{1}{2\alpha}
-
\frac{L_\Phi}{2}
\ge
\frac{L_\Phi}{2}$. Therefore, we obtain
\[
\Phi(\bm\theta^{(k+1)})
\le
\Phi(\bm\theta^{(k)})
-
\frac{L_\Phi}{2}
\left\|
\bm\theta^{(k+1)}
-
\bm\theta^{(k)}
\right\|^2
+
\frac{\alpha}{2}
\left\|
\bm g^{(k)}
-
\nabla_{\bm\theta}\Phi(\bm\theta^{(k)})
\right\|^2 .
\]
Taking conditional expectation and invoking Lemma~\ref{lemma: variance bound}, we obtain
\[
\mathbb E
[
\Phi(\bm\theta^{(k+1)})
\mid
\mathcal F^{(k)}
]
\le
\Phi(\bm\theta^{(k)})
-
\frac{L_\Phi}{2}
\mathbb E
\left[
\left\|
\bm\theta^{(k+1)}
-
\bm\theta^{(k)}
\right\|^2
\mid
\mathcal F^{(k)}
\right]
+
\frac{\alpha}{2}\sigma_g^2.
\]
Taking total expectation gives
\[
\frac{L_\Phi}{2}
\mathbb E
\left[
\left\|
\bm\theta^{(k+1)}
-
\bm\theta^{(k)}
\right\|^2
\right]
\le
\mathbb E[\Phi(\bm\theta^{(k)})]
-
\mathbb E[\Phi(\bm\theta^{(k+1)})]
+
\frac{\alpha}{2}\sigma_g^2.
\]
Invoking the definition of $\widehat G^{(k)}$, we have $\left\|
\bm\theta^{(k+1)}
-
\bm\theta^{(k)}
\right\|^2
=
\left\|
\widehat G^{(k)}
\right\|^2$. Thus, we obtain
\[
\frac{L_\Phi}{2}
\mathbb E
\left[
\left\|
\widehat G^{(k)}
\right\|^2
\right]
\le
\mathbb E[\Phi(\bm\theta^{(k)})]
-
\mathbb E[\Phi(\bm\theta^{(k+1)})]
+
\frac{\alpha}{2}\sigma_g^2.
\]
Summing the previous inequality from $k=0$ to $K-1$, and using this
variation bound to telescope the objective sequence, we obtain
\[
\begin{aligned}
\frac{L_\Phi}{2}
\sum_{k=0}^{K-1}
\mathbb E
\left[
\left\|
\widehat G^{(k)}
\right\|^2
\right]
&\le
\Phi(\bm\theta^{(0)})
-
\Phi(\bm\theta^{(K)})+
\frac{K\alpha}{2}\sigma_g^2 .
\end{aligned}
\]
Because $J$ is continuous on a compact set, it is bounded below by some finite value $J^*$. Similarly, since each $\nu_y$ is continuous on the compact set
$\mathcal H_y$, there exists a finite constant $\nu_{\max}$ such that
\[
|\nu_y(\bm\theta_y)|\le \nu_{\max},
\qquad
\forall \bm\theta_y\in\mathcal H_y,\quad \forall y\in[Y].
\]
Moreover, the effective multipliers are bounded as
$0\le \omega_y^{(k)}\le \bar\lambda$. Hence,
\[
\sum_{y\in[Y]}
\omega_y^{(k)}
\nu_y(\bm\theta_y^{(k)})
\ge
-
Y\bar\lambda\nu_{\max}.
\]
Therefore, the penalized objective $\Phi(\bm\theta^{(k)})$ is uniformly bounded below by $\Phi_*
=
J^*
-
Y\bar\lambda\nu_{\max}
$. Thus, we obtain
\[
\begin{aligned}
\frac{L_\Phi}{2}
\sum_{k=0}^{K-1}
\mathbb E
\left[
\left\|
\widehat G^{(k)}
\right\|^2
\right]
&\le
\left(\Phi(\bm\theta^{(0)})
-
\Phi_*\right)+
\frac{K\alpha}{2}\sigma_g^2 .
\end{aligned}
\]
Multiplying both sides by $2/(L_\Phi K)$, we obtain
\begin{align}
\frac{1}{K}
\sum_{k=0}^{K-1}
\mathbb E
\left[
\left\|
\widehat G^{(k)}
\right\|^2
\right]
\le
\frac{2}{L_\Phi K}
\left(
\Phi(\bm\theta^{(0)})-\Phi_*
\right)
+
\frac{\alpha\sigma_g^2}{L_{\Phi}}.\label{equation: bound for the stochastic projected-gradient mapping}
\end{align}
It remains to relate $\widehat G^{(k)}$ to the deterministic
projected-gradient residual $G^{(k)}$. By the nonexpansiveness of the
projection operator, we obtain
\[
\begin{aligned}
\left\|
G^{(k)}
-
\widehat G^{(k)}
\right\|
&=
\Bigg\|
\Pi_{\mathcal H}
\left(
\bm\theta^{(k)}
-
\alpha\bm g^{(k)}
\right)
-
\Pi_{\mathcal H}
\left(
\bm\theta^{(k)}
-
\alpha
\nabla_{\bm\theta}\Phi(\bm\theta^{(k)})
\right)
\Bigg\|
\le
\alpha\left\|
\bm g^{(k)}
-
\nabla_{\bm\theta}\Phi(\bm\theta^{(k)})
\right\|.
\end{aligned}
\]
Then, by the triangle inequality, the inequality $(a+b)^2\le 2a^2+2b^2$, and the preceding inequality, we obtain
\begin{align*}\left\|
G^{(k)}
\right\|^2&=\left\|
G^{(k)}+\widehat G^{(k)}
-
\widehat G^{(k)}
\right\|^2\le 2\left\|
\widehat G^{(k)}\right\|^2+2\left\|G^{(k)}
-
\widehat G^{(k)}
\right\|^2
\le\\&
2
\left\|
\widehat G^{(k)}
\right\|^2
+
2\alpha^2
\left\|
\bm g^{(k)}
-
\nabla_{\bm\theta}\Phi(\bm\theta^{(k)})
\right\|^2.
\end{align*}
Taking expectations, invoking Lemma~\ref{lemma: variance bound}, and averaging over $k=0,\dots,K-1$, we obtain
\[
\frac{1}{K}
\sum_{k=0}^{K-1}
\mathbb E
\left[
\left\|
G^{(k)}
\right\|^2
\right]
\le
2
\left[
\frac{1}{K}
\sum_{k=0}^{K-1}
\mathbb E
\left[
\left\|
\widehat G^{(k)}
\right\|^2
\right]
\right]
+
2\alpha^2\sigma_g^2.
\]
Substituting the bound for the stochastic projected-gradient residual from~\eqref{equation: bound for the stochastic projected-gradient mapping}, we obtain
\begin{align*}
\frac{1}{K}
\sum_{k=0}^{K-1}
\mathbb E
\left[
\left\|
G^{(k)}
\right\|^2
\right]
&\le
\frac{4}{L_\Phi K}
\left(
\Phi_0(\bm\theta^{(0)})-\Phi_*
\right)
+
\frac{2\alpha\sigma_g^2}{L_{\Phi}}
+
2\alpha^2\sigma_g^2.
\end{align*}
Finally, since the minimum of a finite collection is bounded above by its
average, there exists an iteration $k^*\in\{0,\dots,K-1\}$ such that
\begin{align*}
\mathbb E
\left[
\left\|
G^{(k^*)}
\right\|^2
\right]
&\le
\frac{4}{L_\Phi K}
\left(
\Phi_0(\bm\theta^{(0)})-\Phi_*
\right)
+
\frac{2\alpha\sigma_g^2}{L_{\Phi}}
+
2\alpha^2\sigma_g^2.
\end{align*}
\item By setting $\alpha=\frac{1}{K}$ in the results in part (a) we obtain the desired result.

\end{enumerate}
\end{proof}

\section{Benchmark planning model}
\label{ax:benchmark_model}

Below we report the detailed uncertainty-agnostic benchmark planning model. 
It enforces the same investment constraints and models operations as a single-stage dispatch using realized renewable availability.

\begin{subequations}
\label{bil2}
\begin{alignat}{2}
\text{min}
& \sum_{y \in [Y]} \Big(\frac{1}{1+r}\Big)^{\tau(y)} \Bigg(
    \sum_{l \in [L]} c^{\rm L}_{l,y}\,\bar f_{l,y}
  + \!\sum_{g \in [G]} c^{\rm G}_{g,y}\,\bar p_{g,y}
  + \sum_{b \in [B]} c^{\rm E}_{b,y}\,\bar E_{b,y} + \!\!\sum_{w \in [W]} \!c_{w,y}^{\rm W}\,\bar p^{w}_{w,y}
  \nonumber \hspace{-10cm}\\
&\hspace{3.5cm}
  + \frac{365}{D}\!\sum_{d \in [D]} \sum_{t \in [T]} \Big(
      \sum_{g \in [G]} c^{\rm G}_{g}\, p_{g,d,t,y}
    + \! \sum_{b \in [B]} c^{\rm B}_b\!\left(p^{dis}_{b,d,t,y} + p^{ch}_{b,d,t,y}\right)\nonumber\hspace{-10cm}\\
&\hspace{3.5cm}
    + \sum_{i \in [N]} c^{\rm shed}\, d^{shed}_{i,d,t,y}
    + \sum_{w \in [W]} c^{\rm cur}\, w^{cur}_{w,d,t,y}
    \Big)
  \Bigg)\hspace{-10cm}
\label{objpk}
\\
\text{s.t.}\
&\eqref{Line capacity limits}-\eqref{Energy-to-power ratio.}
\quad \text{[Investment planning constraints]}
\label{eq:benchmark_planning_constraints}
\\
&\sum_{g \in [G]_i} p_{g,d,t,y}
+ \sum_{b \in [B]_i} p^{B}_{b,d,t,y}
+ \sum_{l \in \delta_i^{in}} f_{l,d,t,y}
- \sum_{l \in \delta_i^{out}} f_{l,d,t,y} 
+ \sum_{w \in [W]_i} p^{w}_{w,d,t,y}
+ d^{shed}_{i,d,t,y}
= d_{i,d,t,y} \hspace{-5cm}\label{rt_first_pk}\\[-1.3em]
& &&  \hspace{3cm}\forall i \in [N],\ \forall t \in [T]
\nonumber
\\
&0 \le p_{g,d,t,y} \le \bar p_{g,y}
&&  \forall g \in [G],\ \forall t \in [T]
\label{Generation output limits.rt_pk}
\\
& p_{g,d,t+1,y} - p_{g,d,t,y} \le R^{\rm up}_g
&&  \forall g \in [G],\ \forall t=1,\dots,T-1
\label{ramp1pk}
\\
& p_{g,d,t,y} - p_{g,d,t+1,y} \le R^{\rm dn}_g
&&  \forall g \in [G],\ \forall t=1,\dots,T-1
\label{ramp2_pk}
\\
& f_{l,d,t,y} = \sum_{i \in [N]} H_{l,i}\theta_{i,d,t,y}
&&  \forall l \in [L],\ \forall t \in [T]
\label{angle-pk}
\\
& \theta_{ref,d,t,y} = 0
&&  \forall t \in [T]
\label{angleref-pk}
\\
& -\bar f_{l,y} \le f_{l,d,t,y} \le \bar f_{l,y}
&&  \forall l \in [L],\ \forall t \in [T]
\label{linelimit_pk1}
\\
& 0 \le e_{b,d,t,y} \le \bar E_{b,y}
&&  \forall b \in [B],\ \forall t \in [T]
\label{batt_pk2}
\\
& -\bar p^B_{b,y} \le p^{B}_{b,d,t,y} \le \bar p^B_{b,y},
&&  \forall b \in [B],\ \forall t \in [T]
\label{pkstoragepower}
\\
& e_{b,d,t+1,y} = e_{b,d,t,y} - p^{B}_{b,d,t,y},
&&  \forall b \in [B],\ \forall t=1,\dots,T-1
\label{pkstorageenergy}\\
& 0 \le p^{dis}_{b,d,t,y}
&&  \forall b \in [B],\ \forall t \in [T]
\label{Storage power obj-variable_pk}
\\
& 0 \le p^{ch}_{b,d,t,y}
&&  \forall b \in [B],\ \forall t \in [T]
\label{Storage power obj-variable2_pk}
\\
& p^{B}_{b,d,t,y} = p^{dis}_{b,d,t,y} - p^{ch}_{b,d,t,y}
&&  \forall b \in [B],\ \forall t \in [T]
\label{Storage power obj_pk}
\\
& 0 \le p^{w}_{w,d,t,y} \le \tilde{\xi}_{w,d,t,y}\,\bar p^w_{w,y}
&&  \forall w \in [W],\ \forall t \in [T]
\label{windda_pk}
\\
& 0 \le d^{shed}_{i,d,t,y}
&&  \forall i \in [N],\ \forall t \in [T]
\label{Corrective action limits.pk}
\\
& w^{cur}_{w,d,t,y}
= \tilde{\xi}_{w,d,t,y}\,\bar p^w_{w,y} - p^{w}_{w,d,t,y}
&&  \forall w \in [W],\ \forall t \in [T]
\label{Corrective action limits.curpk}
\\
& \eta_y ( \sum_{g \in [G]} p_{g,d,t,y} + \sum_{w \in [W]} p^w_{w,d,t,y})
\le \sum_{w \in [W]} p^w_{w,d,t,y}
&&  \forall t \in [T]
\label{ren-target_pk}
\end{alignat}
\end{subequations}
Objective function \eqref{objpk} minimizes the total system cost, which consists of the investment cost, the operating cost over all scenario days and time periods.
Constraints \eqref{eq:benchmark_planning_constraints} are the investment planning constraints. 
At the operational level, the same constraints on power balance \eqref{rt_first_pk}, generator constraints \eqref{Generation output limits.rt_pk}--\eqref{ramp2_pk}, power flow \eqref{angle-pk}--\eqref{linelimit_pk1}, battery storage \eqref{batt_pk2}–\eqref{Storage power obj_pk}, wind power generation \eqref{windda_pk}, load shedding \eqref{Corrective action limits.pk}, and wind curtailment \eqref{Corrective action limits.curpk} as discussed in Section~\ref{ss:planningandoperation} above are enforced.
Finally, the constraint \eqref{ren-target_pk} enforces the renewable penetration target.

\FloatBarrier

\bibliographystyle{cas-model2-names}

\bibliography{cas-refs}

\end{document}